\documentclass[a4paper,11pt]{article}

\usepackage[utf8]{inputenc}
\usepackage[british]{babel}

\usepackage[mono=false]{libertine}
\usepackage[T1]{fontenc}

\usepackage{amsthm}
\usepackage[cmintegrals,libertine]{newtxmath}
\usepackage[cal=euler, scr=boondoxo]{mathalfa}
\useosf

\usepackage{microtype}

\usepackage{numprint}

\usepackage[margin=2.5cm]{geometry}
\usepackage{tikz}

\usepackage{graphicx}
\usepackage[font=sf]{caption}
\usepackage{hyperref}
\hypersetup{
	colorlinks=true,
		citecolor=blue!60!black,
		linkcolor=red!60!black,
		urlcolor=green!40!black,
		filecolor=yellow!50!black,
	breaklinks=true,
	pdfpagemode=UseNone,
	bookmarksopen=false,
}

\usepackage{enumitem}

\newtheorem{thm}{Theorem}
\newtheorem{lem}{Lemma}
\newtheorem{prop}{Proposition}
\newtheorem{cor}{Corollary}
\newtheorem{open}{Open question}

\newtheorem{thmA}{Theorem}

\theoremstyle{definition}
\newtheorem{rem}{Remark}

\newcommand{\ensembles}[1]{\mathbb{#1}}
	\newcommand{\N}{\ensembles{N}}
	\newcommand{\Z}{\ensembles{Z}}
	\newcommand{\R}{\ensembles{R}}

\newcommand{\ind}[1]{\mathbb{I}_{\{#1\}}}
	\renewcommand{\P}{\ensembles{P}}
	\newcommand{\E}{\ensembles{E}}

\renewcommand{\Pr}[1]{\P\left(#1\right)}
\newcommand{\Prc}[2]{\P\left(#1 \;\middle|\; #2\right)}
\newcommand{\Es}[1]{\E\left[#1\right]}
\newcommand{\Esc}[2]{\E\left[#1 \; \middle|\; #2\right]}
\newcommand{\Var}[1]{\mathrm{Var}\left(#1\right)}

\renewcommand{\d}{\mathrm{d}}
\newcommand{\ex}{\mathrm{e}}
\renewcommand{\i}{\mathrm{i}}
\newcommand{\q}{\mathbf{q}}

\newcommand{\cadlag}{c\`adl\`ag}

\newcommand{\Ballfpp}[1][t]{\mathrm{Ball}^{\mathrm{fpp}}_{#1}}

\newcommand{\hBall}[1][r]{\overline{\mathrm{Ball}}\vphantom{B}_{#1}}
\newcommand{\hBallfpp}[1][t]{\overline{\mathrm{Ball}}\vphantom{B}^{\mathrm{fpp}}_{#1}}
\newcommand{\hBallEden}[1][r]{\overline{\mathrm{Ball}}\vphantom{B}^{\mathrm{Eden}}_{#1}}

        \newcommand{\map}{\mathfrak{m}}

        \newcommand{\Map}{\mathfrak{M}}

\newcommand{\cv}[1][n]{\quad\xrightarrow[#1 \to \infty]{}\quad}
\newcommand{\cvdist}[1][n]{\quad\xrightarrow[#1 \to \infty]{(d)}\quad}
\newcommand{\cvas}[1][n]{\quad\xrightarrow[#1 \to \infty]{(a.s.)}\quad}
\newcommand{\cvproba}[1][n]{\quad\xrightarrow[#1 \to \infty]{(\P)}\quad}
\newcommand{\cvL}[1][n]{\quad\xrightarrow[#1 \to \infty]{(\mathrm{L}^1)}\quad}

\title{Infinite random planar maps related to Cauchy processes}
\DeclareSymbolFont{extraup}{U}{zavm}{m}{n}
\DeclareMathSymbol{\vardspade}{\mathalpha}{extraup}{81}
\DeclareMathSymbol{\varheart}{\mathalpha}{extraup}{86}
\DeclareMathSymbol{\vardiamond}{\mathalpha}{extraup}{87}
\DeclareMathSymbol{\varclub}{\mathalpha}{extraup}{84}

\makeatletter
\renewcommand*{\@fnsymbol}[1]{\ensuremath{\ifcase#1\or  \vardspade \or \vardiamond \or \varheart\or \varclub \or
   \mathsection\or \mathparagraph\or \|\or **\or \dagger\dagger
   \or \ddagger\ddagger \else\@ctrerr\fi}}
\makeatother

\author{
	Timothy \textsc{Budd}\thanks{Institut de Physique Théorique, CEA, Université Paris-Saclay. \hfill  \href{mailto:timothy.budd@cea.fr}{\texttt{timothy.budd@cea.fr}}} 
\qquad\&\qquad
	Nicolas \textsc{Curien}\thanks{D\'epartement de Math\'ematiques, Univ. Paris-Sud, Universit\'e Paris-Saclay and IUF.\hfill  \href{mailto:nicolas.curien@gmail.com}{\texttt{nicolas.curien@gmail.com}}} 
\qquad\&\qquad
	Cyril \textsc{Marzouk}\thanks{D\'epartement de Math\'ematiques, Univ. Paris-Sud, Universit\'e Paris-Saclay.\hfill  \href{mailto:cyril.marzouk@math.u-psud.fr}{\texttt{cyril.marzouk@math.u-psud.fr}}}
}

\begin{document}

\maketitle

\begin{abstract}
We study the geometry of infinite random Boltzmann planar maps having weight of polynomial decay of order $k^{-2}$ for each vertex of degree $k$. These correspond to the dual of the discrete ``stable maps'' of Le Gall and Miermont \cite{Le_Gall-Miermont:Scaling_limits_of_random_planar_maps_with_large_faces} studied in \cite{Budd-Curien:Geometry_of_infinite_planar_maps_with_high_degrees} related to a symmetric Cauchy process, or alternatively to the maps obtained after taking the gasket of a critical $O(2)$-loop model on a random planar map. We show that these maps have a striking and uncommon geometry. In particular we prove that the volume of the ball of radius $r$ for the graph distance has an intermediate rate of growth and scales as $\ex^{\sqrt{r}}$. We also perform first passage percolation with exponential edge-weights and show that the volume growth for the fpp-distance scales as $\ex^{r}$. Finally we consider site percolation on these lattices: although percolation occurs only at $p=1$, we identify a phase transition at $p=1/2$ for the length of interfaces. On the way we also prove new estimates on random walks attracted to an asymmetric Cauchy process.
\end{abstract}

\begin{figure}[!ht]\centering
\includegraphics[width=.7\linewidth]{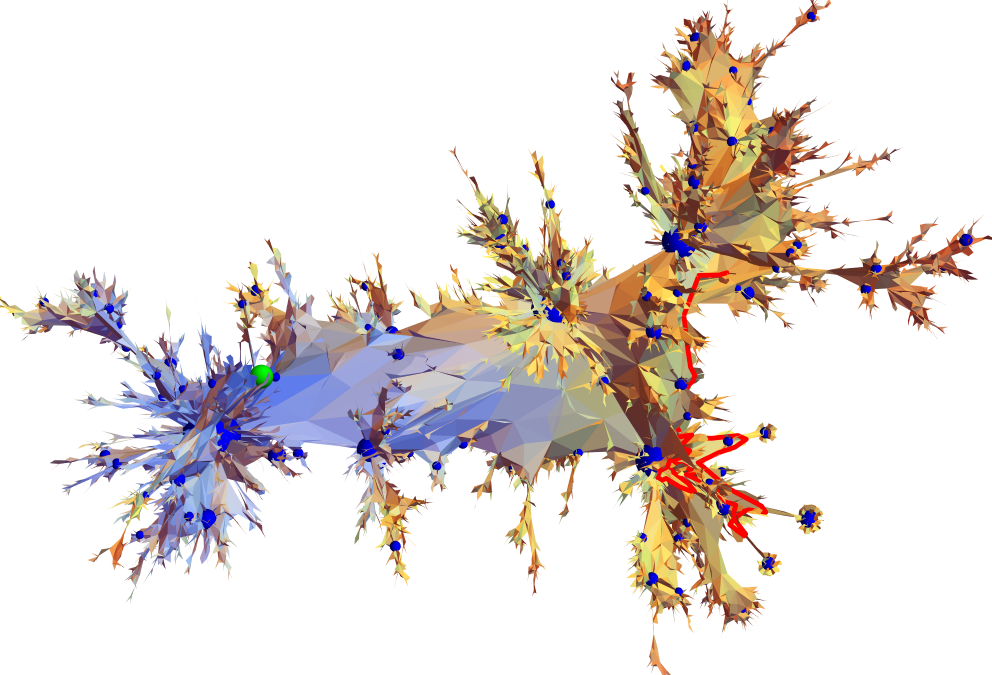}
\caption{A representation of a geodesic ball in a randomly sampled infinite map considered in this paper. The blue balls indicate high degree vertices, the green ball is the root-vertex, while the red curve indicates the boundary of the ball.}
\end{figure}

\section{Introduction}
This work studies the geometry of random Boltzmann planar maps with high degrees in the continuation of \cite{Budd-Curien:Geometry_of_infinite_planar_maps_with_high_degrees} and focuses on the critical case $a=2$ which was left aside there. The geometry of those maps turns out to involve in an intricate way random walks whose step distribution is in the domain of attraction of the standard symmetric Cauchy process and displays an unexpected large scale geometry such as an intermediate rate of growth. Let us first present rigorously the model of random maps we are dealing with. To stick to the existing literature we prefer to introduce the model of maps with large faces and then take their dual maps to get maps with large vertex degrees (as opposed to dealing with maps with large vertex degrees directly).

\paragraph{Boltzmann planar maps with high degrees.}  
In all this work, we consider rooted planar maps, i.e. graphs embedded in the two-dimensional sphere or in the plane, and equipped  with a distinguished oriented edge, the \emph{root-edge} of the map, whose origin vertex is the \emph{root-vertex} of the map, and the face adjacent to the right of the root-edge is called the \emph{root-face}. For technical reasons, we restrict ourselves to \emph{bipartite} maps (all faces have even degree), we denote by $\mathcal{M}$ the set of all such finite maps. Given a non-zero sequence $ \mathbf{q}= (q_{k})_{ k \geq 1}$ of non-negative numbers we define a measure $w$ on $ \mathcal{M}$ by putting for each $ \map \in \mathcal{M}$:
\[w(\map) \enskip=\enskip \prod_{f\in \mathsf{Faces}(\map)} q_{ \deg(f)/2}.\]
When the total mass $w ( \mathcal{M})$ is finite we say that $ \mathbf{q}$ is \emph{admissible} and we can normalise $w$ into a probability measure. This model of random planar maps was considered in \cite{Marckert-Miermont:Invariance_principles_for_random_bipartite_planar_maps}. We shall also further assume that the weight sequence $ \mathbf{q}$ is \emph{critical} in the sense of \cite{Marckert-Miermont:Invariance_principles_for_random_bipartite_planar_maps} (see also \cite{Budd:The_peeling_process_of_infinite_Boltzmann_planar_maps,Curien:Peccot}), which according to \cite{Bernardi-Curien-Miermont:A_Boltzmann_approach_to_percolation_on_random_triangulations} is equivalent to having 
\[\sum_{  \map \in \mathcal{M}} w( \map) \# \mathsf{Vertices}( \map)^{2} = \infty.\]
Under these conditions, one can define a random infinite planar map $ \Map_{\infty}$ of the plane as the local limit as $n \to \infty$ of random Boltzmann maps distributed according to $w$ and conditioned on having $n$ vertices (along the sequence of integers for which this makes sense), see \cite{Bjornberg-Stefansson:Recurrence_of_bipartite_planar_maps, Stephenson:Local_convergence_of_large_critical_multi_type_Galton_Watson_trees_and_applications_to_random_maps, Curien:Peccot}; the map $\Map_\infty$ is almost surely locally finite and one-ended. As in  \cite{Budd-Curien:Geometry_of_infinite_planar_maps_with_high_degrees,Le_Gall-Miermont:Scaling_limits_of_random_planar_maps_with_large_faces} we shall henceforth restrict ourself to weight sequences of the form 
\begin{equation}\label{eq:asymptotic_q}
q_k \enskip\sim\enskip \mathsf{p}_\q \cdot c_\q^{-k+1} \cdot k^{-a} \qquad\text{as}\qquad k \to \infty,
\end{equation}
with $a \in (3/2, 5/2)$, for some constants $c_\mathbf{q},  \mathsf{p}_{ \mathbf{q}}>0$.  Note the above admissibility and criticality conditions impose a fine tuning of the weights $q_{k}$ but such sequences do exist as proved in \cite{Le_Gall-Miermont:Scaling_limits_of_random_planar_maps_with_large_faces}, see also the explicit sequences provided in \cite[Section 6]{Budd-Curien:Geometry_of_infinite_planar_maps_with_high_degrees}. As in \cite{Budd-Curien:Geometry_of_infinite_planar_maps_with_high_degrees}, the object of study in the present work is not the map $  \Map_{\infty}$ itself but rather its dual map 
\[\Map_{\infty}^{\dagger},\]
where the roles of vertices and faces are exchanged (the root-edge of $\Map_\infty^\dagger$ is the dual of the root-edge of $ \Map_{\infty}$ oriented to cross it from right to left, so the root-vertex of $\Map_{\infty}^{\dagger}$ corresponds to the root-face of $\Map_{\infty}$). With our assumptions on the weight sequence $ \mathbf{q}$, the random map $ \Map_{\infty}^{\dagger}$ thus has vertices of very large degree. The geometry of $ \Map_{\infty}^{\dagger}$ displays a phase transition depending on the value $a \in (3/2;5/2)$: for $a \in (3/2;2)$ --the so-called \emph{dense phase}-- the volume growth of the map is exponential in the radius whereas for $a \in (2;5/2)$ --the so-called \emph{dilute phase}-- it is polynomial of exponent $ \frac{a-1/2}{a-2}$, see \cite{Budd-Curien:Geometry_of_infinite_planar_maps_with_high_degrees}. The critical case $a=2$ was left aside in \cite{Budd-Curien:Geometry_of_infinite_planar_maps_with_high_degrees} and is precisely the object of our current investigation!

\paragraph{Results.} For $r \geq 0$ we denote by $ \mathrm{Ball}_{r}( \Map_{\infty}^{\dagger})$ the sub-map of $ \Map_{\infty}^{\dagger}$ obtained by keeping only the vertices which are at distance at most $r$ from the root-vertex of $\Map_{\infty}^{\dagger}$ and we consider its hull 
\[\overline{\mathrm{Ball}}_{r}( \Map_{\infty}^{\dagger})\]
made by adding to $ \mathrm{Ball}_{r}(  \Map_{\infty}^{\dagger})$ all the finite connected components of its complement in $ \Map_{\infty}^{\dagger}$. We define the volume $\|  \mathfrak{m}\|$ of a map as being the number of faces of the map.
Our main result shows that the volume growth of $  \Map_{\infty}^{\dagger}$ is intermediate between exponential and polynomial:

\begin{thmA}[Graph volume growth]\label{thm:intro_dual} If $ \mathbf{q}$ is admissible, critical and satisfies \eqref{eq:asymptotic_q} with $a =2$ then we have 
  \begin{eqnarray*}\frac{\log \| \overline{\mathrm{Ball}}_{r}( \Map_{\infty}^{\dagger})\|}{ \sqrt{r}}  & \xrightarrow[r\to\infty]{ (\mathbb{P})} & 
  \frac{3 \pi}{\sqrt{2}}.  \end{eqnarray*}
\end{thmA}
Our results also imply that the number of edges adjacent to the hull of the ball of radius $r$ is of order $ \ex^{  \pi \sqrt{2r}}$, see Theorem \ref{thm:graph_distance_growth}. It is maybe surprising to notice that the first order of the growth of $\| \overline{\mathrm{Ball}}_{r}( \Map_{\infty}^{\dagger})\|$ in fact does not depend on the parameter $ \mathsf{p}_{ \mathbf{q}}$ introduced in \eqref{eq:asymptotic_q} --a similar phenomenon was already observed in the dense phase $a \in (3/2,2)$ in \cite[Theorem 5.3]{Budd-Curien:Geometry_of_infinite_planar_maps_with_high_degrees}. However, the scaling constant $ \mathsf{p}_{ \mathbf{q}}$ does appear if one modifies the metric as follows: Consider the first-passage percolation (fpp) model on $ \Map_{\infty}^{\dagger}$ where each edge gets an independent weight distributed as an exponential law of parameter $1$ (this model is sometimes referred to as the Eden model on $ \Map_{\infty}$). These weights are interpreted as random lengths for the edges of $ \Map_{\infty}^{\dagger}$ and give rise to the associated fpp-distance. If $\Ballfpp[r]$ and $\hBallfpp[r]$ are the associated ball and hull of radius $r \geq 0 $ we can prove:
\begin{thmA}[Eden volume growth]\label{thm:intro_Eden} If $ \mathbf{q}$ is admissible, critical and satisfies \eqref{eq:asymptotic_q} with $a =2$ then we have 
 \begin{eqnarray*}\frac{\log \|\hBallfpp[r](\Map_\infty^{\dagger})\|}{r}  & \xrightarrow[r\to\infty]{( \mathbb{P})} & \frac{3}{2} \pi^2 \mathsf{p}_\q.  \end{eqnarray*}
\end{thmA}

Again, our result also includes the perimeter, see Theorem \ref{thm:perimeter_volume_fpp}. Comparing the two theorems, we see that the graph ball of radius $r$ has a volume growth of type $\ex^{ \sqrt{r}}$ whereas the fpp-ball of radius $r$ has exponential volume growth. This indeed shows that large degree vertices dramatically affect the geometric structure by diminishing the distances due to the low weights on their edges creating short-cuts in the graph. A similar and even more dramatic effect was observed in \cite{Budd-Curien:Geometry_of_infinite_planar_maps_with_high_degrees} where the fpp-distance to infinity becomes finite in the dense case $ a \in (3/2,2)$. The case $a=2$ we consider here is thus in between the dense case $a \in (3/2,2)$ and the dilute case $a \in (2,5/2)$ where the fpp-distance and graph metric are believed to be proportional at large scales, as is the case for random triangulations \cite{Curien-Le_Gall:First_passage_percolation_and_local_modifications_of_distances_in_random_triangulations}.

\begin{rem}
In the two previous theorems, we have considered the volume of a map as its number of faces; we may as well consider its number of vertices, the results would be exactly the same, see Remark \ref{rem:general_peeling_vertices_faces} below (written for $\Map_\infty$ as opposed to its dual here).
\end{rem}

We also study Bernoulli site percolation on the random graphs $\Map_{\infty}^{\dagger}$ with parameter $p \in [0,1]$. Although there is no usual phase transition for the existence of an infinite cluster because these graphs almost surely have infinitely many cut-points, we exhibit a phase transition at $p= \frac{1}{2}$ for the size distribution of the origin cluster. More precisely, we consider face percolation on $ \Map^{(\infty)}$, the half-plane version of $ \Map_{\infty}$, and show that the parameter $ p=1/2$ is singular for the tail of the length of percolation interfaces, see Proposition \ref{prop:number_edges_boundary_perco_half_plane} for details.

\paragraph{Connections with $O(n)$-loop models and SLE/CLE.} As argued in \cite{Le_Gall-Miermont:Scaling_limits_of_random_planar_maps_with_large_faces, Borot-Bouttier-Guitter:A_recursive_approach_to_the_O_n_model_on_random_maps_via_nested_loops}, one natural way to get Boltzmann planar maps with a critical weight sequence $ \mathbf{q}$ satisfiying \eqref{eq:asymptotic_q} with $ a \in (3/2; 5/2)$ is to consider the so-called gasket of a critical loop-decorated planar quadrangulation which is obtained by pruning the interior of all the outer-most loops in the map. See \cite{Chen-Curien-Maillard:The_perimeter_cascade_in_critical_Boltzmann_quadrangulations_decorated_by_an_O_n_loop_model} for a study of the perimeter cascades in such maps and \cite{Budd:The_peeling_process_on_random_planar_maps_coupled_to_an_O_n_loop_model} for a peeling approach to these objects. Although the precise case we are dealing with in these lines seems excluded in \cite{Borot-Bouttier-Guitter:A_recursive_approach_to_the_O_n_model_on_random_maps_via_nested_loops}, heuristically the case $a=2$ corresponds to the $O(n)$ model in the critical case $n=2$.  In particular, folklore conjectures state that, once conformally uniformised, these decorated maps should converge towards the CLE$_{\kappa}$ ensembles \cite{Sheffield-Werner:Conformal_loop_ensembles_the_Markovian_characterization_and_the_loop_soup_construction} which are (conjecturally) the limits of the critical $O(n)$-loop models on Euclidean lattices, with the relation $a-2= \frac{4}{\kappa}-1$ with $\kappa \in (8/3,8)$. Our case $a =2$ thus corresponds to the critical case $\kappa = 4$ as expected.

\paragraph{Techniques.} Our main tool is the \emph{edge-peeling process} defined in \cite{Budd:The_peeling_process_of_infinite_Boltzmann_planar_maps} and already used in \cite{Budd-Curien:Geometry_of_infinite_planar_maps_with_high_degrees} in the case when $a \in (3/2,5/2)\backslash\{2\}$. More precisely, the peeling process is an algorithmic way to discover the map and using different algorithms in the process leads to a different understanding of geometric properties of the random map. A common feature of all the peeling processes of  $ \Map_{\infty}$ is that the perimeter and volume growth have the same law independently of the peeling algorithm used. In our case the perimeter process, once rescaled in time and space by $n^{-1}$, converges towards the \emph{symmetric Cauchy process conditioned to stay positive}, see Theorem \ref{thm:convergence_processes_perimeter_volume}. 

As in \cite{Budd-Curien:Geometry_of_infinite_planar_maps_with_high_degrees}, we use the so-called peeling by layers to study the graph metric (Theorem \ref{thm:graph_distance_growth}), the uniform peeling to study the first-passage percolation (Theorem \ref{thm:perimeter_volume_fpp}) and the peeling along percolation interfaces to study percolation (Theorem \ref{thm:cut_edges_half_plane}). The arguments used in this paper are then, on a first glance, very close to those used in \cite{Budd-Curien:Geometry_of_infinite_planar_maps_with_high_degrees}. However, the critical case $a=2$ and the singular characteristics of Cauchy processes make the arguments much more subtle and witness the appearance of totally new phenomena. In passing we also prove several estimates on random walks converging towards (a symmetric or) an asymmetric Cauchy process. Although the literature on random walks converging towards stable L\'evy processes is abundant, it seems that this case is often put aside. Thus, our work may also be viewed as a contribution to the study of random walks and Cauchy processes! In particular we prove: 

\begin{prop}\label{prop:intro_tail_first_ladder_time_Cauchy_walk}
Let $W$ be a random walk on $\Z$ started from $0$ with i.i.d. steps and let $\tau = \inf\{n \ge 1 : W_n \le -1\}$. Suppose that there exist $c_+, c_- > 0$, such that
\[\Pr{W_1>k} \sim \frac{c_+}{k},
\qquad\text{and}\qquad
\Pr{W_1<-k} \sim \frac{c_-}{k}
\qquad\text{as}\qquad k \to \infty.\]
\begin{enumerate}[ref={\theprop(\roman*)}]
	\item\label{prop:tail_first_ladder_time_Cauchy_walk_positive_tail} If $c_+ > c_-$ then $\P(\tau \ge n) = (\log n)^{O(1)}$.
	\item\label{prop:tail_first_ladder_time_Cauchy_walk_negative_tail} If $c_+ < c_-$ then $\P(\tau \ge n) = n^{-1}(\log n)^{O(1)}$.
	\item\label{prop:tail_first_ladder_time_Cauchy_walk_symmetric} If $c_+ = c_-$ and the limit $b = \lim_{n\to\infty}\Es{\frac{W_1}{1+(W_1/n)^2}}$ exists and is finite, then
	\[ \P(\tau \ge n) = n^{-\rho + o(1)}\qquad\text{where}\qquad \rho = \frac{1}{2}+\frac{1}{\pi}\arctan\left(\frac{b}{\pi c_+}\right). \]
\end{enumerate}
\end{prop}

In Remark \ref{rem:polylog} in Section \ref{sec:RW_Cauchy} below, we discuss more precisely the $O(1)$ in the first two cases. We end this introduction by an open question:
\begin{open}
Is the simple random walk on $\Map_\infty^\dagger$ transient?
\end{open}

\noindent \textbf{Acknowledgments:} We acknowledge the support of the Agence Nationale de la Recherche via the grants ANR Liouville (ANR-15-CE40-0013) and ANR GRAAL (ANR-14-CE25-0014). We thank Lo\"ic Chaumont, Ron Doney and Vladimir Vatutin for help with the literature about random walks with Cauchy-type tails. 

C. M. acknowledges support from the Fondation Math\'ematique Jacques Hadamard, the Fondation Sciences Math\'ematiques de Paris and the Universit\'e Pierre et Marie Curie, where this work was initiated.

T. B. acknowledges support from the Fondation Math\'ematique Jacques Hadamard, ERC-Advance grant 291092 “Exploring the Quantum Universe” (EQU), and the Niels Bohr Institute, University of Copenhagen, where this work was initiated.

\begin{center}
\begin{minipage}{.9\linewidth}
\emph{From now on we fix once and for all an admissible and critical weight sequence $ \mathbf{q}$ satisfying \eqref{eq:asymptotic_q} for the exponent $a=2$ and denote by $ \Map_{\infty}$ the associated infinite random Boltzmann planar map. Moreover, we assume that $\mathbf{q}$ is not supported by any sub-lattice of $\Z$ to avoid complication.}
\end{minipage}  
\end{center}

\paragraph{Notation.}
We shall write $\Z_+$ for $\{0, 1, \dots\}$, $\Z_-$ for $\{\dots, -1, 0\}$ and $\N$ for $\{1, 2, \dots\}$. For every set $A$, we denote by $\mathbb{I}_A$ the indicator function of $A$. For a real-valued sequence $(x_n)_{n \ge 0}$, we put $\Delta x_n \coloneqq x_{n+1} - x_n$ for all $n \ge 0$.

\section{\texorpdfstring{Scaling limit of the peeling process on $\Map_\infty$}%
	{Scaling limit of the peeling process on M\_infinity}}
\label{sec:peeling_scaling_limits}

In this section we recall the edge-peeling process introduced in \cite{Budd:The_peeling_process_of_infinite_Boltzmann_planar_maps} and the connection with a random walk in the domain of attraction of the symmetric Cauchy process. We then prove an invariance principle for the perimeter and volume of a general peeling process. Later, we shall derive Theorems \ref{thm:intro_dual} and \ref{thm:intro_Eden} by applying these general results to well-chosen peeling procedures. The presentation and notation is inspired by \cite{Budd-Curien:Geometry_of_infinite_planar_maps_with_high_degrees,Curien:Peccot}, to which we refer for more details.

\subsection{\texorpdfstring{Filled-in explorations in $\Map_\infty$}%
	{Filled-in explorations in M\_infinity}}
\label{sec:def_peeling_general}

Recall that we consider rooted bipartite planar maps; such maps will be denoted by $\map$ or $\Map$ (keeping in mind that we are in fact interested in their duals  $\map^\dagger$ or $\Map^\dagger$) and that the root-face of  $\mathfrak{m}$ is the face adjacent on the right to the root-edge.

Let $ \mathfrak{m}$ be an infinite one-ended bipartite map (an assumption that $ \Map_{\infty}$ satisfies almost surely) a  \emph{peeling exploration}\footnote{Filled-in exploration in the language of \cite[Section 3.1.4]{Curien:Peccot}.} of $\map$ is an increasing sequence $(\overline{\mathfrak{e}}_i)_{i \ge 0}$ of sub-maps of $\map$ containing the root-edge such that $  \overline{\mathfrak{e}}_{i}$ has a distinguished simple face called the hole. By sub-map $ \overline{\mathfrak{e}}_{i} \subset  \mathfrak{m}$ we mean that we can recover $ \map$ from $ \overline{ \mathfrak{e}}_{i}$ by gluing inside the unique hole of $  \overline{ \mathfrak{e}}_{i}$ a bipartite planar map with a (not necessarily simple) boundary of perimeter matching that of the hole of $  \overline{\mathfrak{e}}_{i}$ (this map is uniquely defined). More precisely, a peeling exploration depends on an algorithm $ \mathcal{A}$ which associates with each sub-map $ \overline{ \mathfrak{e}}$ an edge on the boundary of its hole.\footnote{The algorithm can be deterministic or random, but in the latter case the randomness involved must be independent of the un-revealed part of the map.}  Then the peeling exploration of $ \mathfrak{m}$ with algorithm $ \mathcal{A}$ is the following sequence $(\overline{ \mathfrak{e}}_{i})_{i \ge 0}$ of sub-maps of $ \mathfrak{m}$. First $\overline{\mathfrak{e}}_0$ consists only of two simple faces with the same degree as the root-face of $\map$ and an oriented edge, the hole is the face on the left of this root-edge. Then for each $i \ge 0$, given $\overline{\mathfrak{e}}_i$, the sub-map $ \overline{ \mathfrak{e}}_{i+1}$ is obtained by peeling the edge $ \mathcal{A}( \overline{ \mathfrak{e}}_{i})$ in $ \mathfrak{m}$. When peeling an edge, there are two cases depicted in Figure \ref{fig:peeling_dual}:
\begin{enumerate}
\item\label{item:def_peeling_new_face} either the face in $  \mathfrak{m}$ on the other side of $ \mathcal{A}( \overline{ \mathfrak{e}}_{i})$ is not already present in $  \overline{ \mathfrak{e}}_{i}$; in this case $ \overline{ \mathfrak{e}}_{i+1}$ is obtained by adding this face to $ \overline{ \mathfrak{e}}_{i}$ glued onto $ \mathcal{A}( \overline{ \mathfrak{e}}_{i})$ and without performing any other identification of edges,
\item\label{item:def_peeling_gluing} or the other side of $ \mathcal{A}( \overline{ \mathfrak{e}}_{i})$ in $  \mathfrak{m}$ actually corresponds to a face already discovered in  $ \overline{ \mathfrak{e}}_{i}$. In this case $ \overline{ \mathfrak{e}}_{i+1}$ is obtained by performing the identification of the two edges in the hole of $ \overline{ \mathfrak{e}}_{i}$. This usually creates two holes, but since $ \mathfrak{m}$ is one-end, we decide to fill-in the one containing a finite part of $  \mathfrak{m}$. 
\end{enumerate}

\begin{figure}[!ht]
\begin{center}
\includegraphics[width=.8\linewidth]{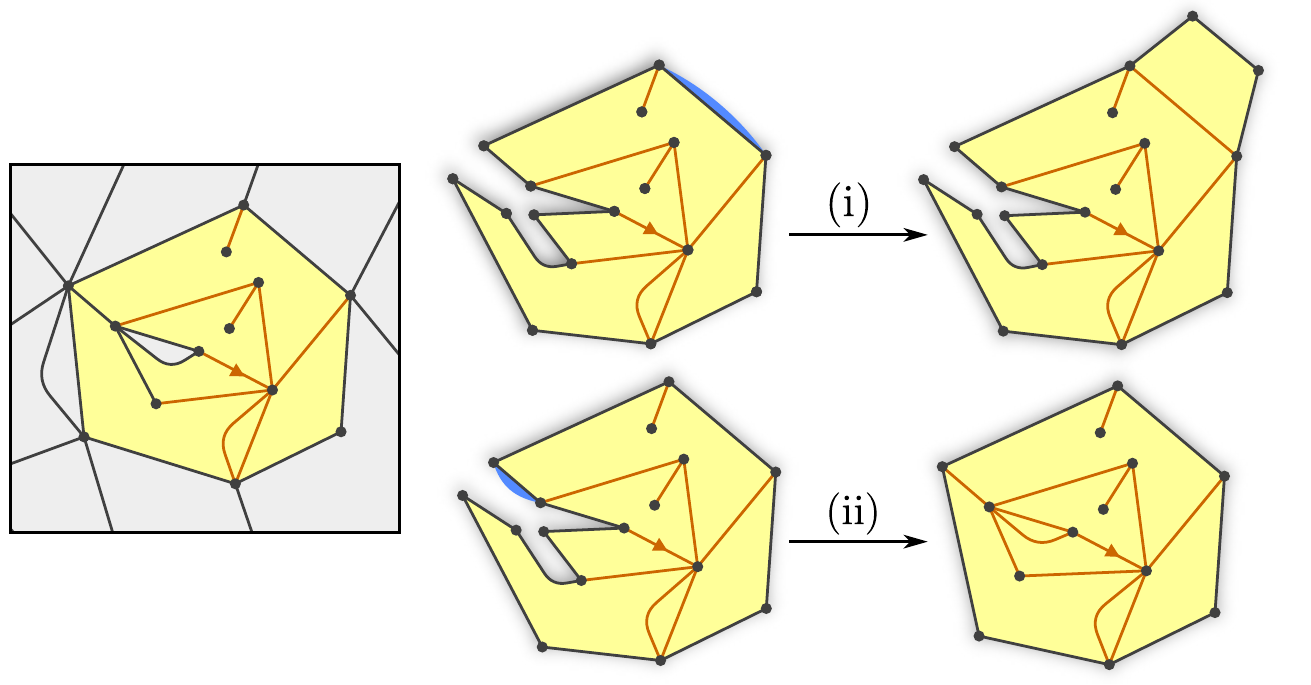}
\caption{Illustration of a (filled-in) peeling step in a one-ended bipartite map. The peeling edge is depicted in blue. In the first case we add a new face adjacent to this edge, and in the second case we identify two edges on the boundary of the hole; in the latter case, the hole is split into two components and we fill-in the finite one.
\label{fig:peeling_dual}}
\end{center}
\end{figure}

\subsection{The perimeter as a Doob transform}
\label{sec:Doob_transform}

We now recall the law of a peeling exploration of $ \Map_{\infty}$. See \cite{Budd:The_peeling_process_of_infinite_Boltzmann_planar_maps, Budd-Curien:Geometry_of_infinite_planar_maps_with_high_degrees, Curien:Peccot} for details.  By definition, if  $ \overline{ \mathfrak{e}} \subset \mathfrak{m}$ is a sub-map, then the perimeter $|\partial  \overline{ \mathfrak{e}}|$ is the number of edges on the boundary of its simple hole whereas its volume $|   \overline{ \mathfrak{e}}|$ is its number of inner vertices, i.e. the vertices not incident to the hole. To describe the transition probabilities of the Markov chain $( \overline{ \mathfrak{e}}_{i})_{i \geq 0}$ we need some enumeration results, see \cite[Section 3.3]{Borot-Bouttier-Guitter:A_recursive_approach_to_the_O_n_model_on_random_maps_via_nested_loops}, \cite[Section 2]{Le_Gall-Miermont:Scaling_limits_of_random_planar_maps_with_large_faces} and \cite{Budd:The_peeling_process_of_infinite_Boltzmann_planar_maps,Curien:Peccot}. 

For every $n, \ell \ge 1$, we denote by $\mathcal{M}^{(\ell)}_n$ the set of all rooted bipartite planar maps with $n$ vertices, such that the root-face (the face adjacent on the right of the root-edge) has degree $2 \ell$. We put $ \mathcal{M}^{(\ell)}= \bigcup_{n \geq 1} \mathcal{M}^{(\ell)}_{n}$. Observe that any bipartite map can be seen as a planar map with a root-face of degree $2$ by simply ``unzipping'' the root-edge. We shall always implicitly use this identification if necessary. We let
\[W_n^{(\ell)} = \sum_{\map \in \mathcal{M}^{(\ell)}_n} \prod_{\substack{f \in \mathsf{Faces}(\map) \\ f \ne \text{ root-face}}} q_{\deg(f)/2},
\qquad\text{and}\qquad
W^{(\ell)} = \sum_{n \ge 0} W_n^{(\ell)},\]
be the masses given respectively to the sets $\mathcal{M}_n^{(\ell)}$ and $\mathcal{M}^{(\ell)}$ by the measure $w$ defined in the introduction, with the slight change that we do not count the root-face anymore. 
Since $\q$ is admissible, $W^{(\ell)}$ is finite for every $\ell \ge 1$ and we let
\begin{equation}\label{eq:def_free_Boltzmann_law_with_boundary}
\P^{(\ell)}(\map) \coloneqq \frac{w(\map)}{W^{(\ell)}}
\qquad\text{for}\qquad
\map \in \mathcal{M}^{(\ell)}
\end{equation}
be the law of a finite Boltzmann rooted bipartite map with a root-face of degree $2\ell$. For $n, \ell \ge 1$, we may define a law $\P_n^{(\ell)}$ on $\mathcal{M}_n^{(\ell)}$ in a similar way; as recalled in the introduction, for each $\ell \ge 1$ fixed, these laws converge weakly as $n \to \infty$ for the local topology. We shall denote by $\P_\infty^{(\ell)}$ the limit, it is a distribution on the set of infinite rooted bipartite maps with a root-face of degree $2\ell$ and which are one-ended. With this notation, $\Map_\infty$ has law $\P_\infty^{(1)}$.

Recall that we assume that the weight sequence $\q$ is admissible and critical and satisfies \eqref{eq:asymptotic_q} with $a=2$, which imposes $c_\q$ and $\mathsf{p}_\q$ to be fine-tuned; we have then \cite[Chapter 5.1.3]{Curien:Peccot}
\[W^{(\ell)} \sim \frac{\mathsf{p}_\q}{2} \cdot c_\q^{\ell+1} \cdot \ell^{-2}
\qquad\text{as}\qquad \ell \to \infty.\]
As noticed in \cite{Budd:The_peeling_process_of_infinite_Boltzmann_planar_maps} the following function on $\Z$ 
\begin{equation}\label{eq:def_h_up}
h^\uparrow(\ell) \coloneqq 2\ell \cdot 2^{-2\ell} \cdot \binom{2\ell}{\ell} \cdot \ind{\ell \ge 1}
\end{equation}
will play a crucial role in connection with the probability measure $\nu$ on $\Z$ defined by
\begin{equation}\label{eq:def_nu}
\nu(k) \coloneqq
\begin{cases}
q_{k+1} \cdot c_\q^k & \text{for } k \geq 0 \\
2 \cdot W^{(-1-k)} \cdot c_\q^k & \text{for } k \leq -1.
\end{cases}
\end{equation}
The fact that $\nu$ is a probability distribution follows from the admissibility of the weight sequence $ \mathbf{q}$, see \cite[Lemma 9]{Curien:Peccot}. Notice also that $\nu$ characterises the weight sequence $ \mathbf{q}$ and that the previous asymptotic behaviour of $W^{(\ell)}$ together with \eqref{eq:asymptotic_q} yields
\begin{equation}\label{eq:tail_nu}
\nu(-k) \sim \nu(k) \sim \mathsf{p}_\q \cdot k^{-2} \qquad\text{as}\qquad k \to \infty.
\end{equation}
Furthermore, since $ \mathbf{q}$ is critical,  the function $h^\uparrow$ is (up to a multiplicative constant) the only non-zero harmonic function on $\N$ for the random walk with independent increments distributed according to $\nu$ (we say that $h^\uparrow$ is $\nu$-harmonic at these points) and that vanishes on $\Z_-$.\footnote{In fact, according to \cite[Section 3.2]{Budd:The_peeling_process_of_infinite_Boltzmann_planar_maps} the sequence $\q$ is admissible and critical \emph{if and only if} $\nu$ is a probability distribution and $h^\uparrow$ is $\nu$-harmonic on $\N$.} Let $S$ be a random walk with i.i.d. increments of law $\nu$ given in \eqref{eq:def_nu}; we may define its version $S^\uparrow$ conditioned to never hit $\Z_-$ via a Doob $h^\uparrow$-transform.

We now recall from \cite{Budd:The_peeling_process_of_infinite_Boltzmann_planar_maps} the law of a peeling exploration  $(\overline{\mathfrak{e}}_i)_{i \ge 0}$ of $\Map_\infty$ (for a fixed peeling algorithm $ \mathcal{A}$) as described in the previous subsection. Recall the two possible outcomes \ref{item:def_peeling_new_face} and \ref{item:def_peeling_gluing} from Section \ref{sec:def_peeling_general} when peeling an edge on the boundary of $ \overline{ \mathfrak{e}}_{i}$. Conditionally on the current exploration $\overline{\mathfrak{e}}_i$ and on the selected edge to peel $ \mathcal{A}( \overline{ \mathfrak{e}}_{i})$, if the perimeter of $\overline{\mathfrak{e}}_i$ is $2\ell$ for some $\ell \ge 1$, then the peeling of $ \mathcal{A}( \overline{ \mathfrak{e}}_{i})$ leads to a new face of degree $2k$ with probability
\begin{equation}\label{eq:qp}
p_{k}^{(\ell)} \coloneqq \nu(k-1) \frac{h^\uparrow(\ell+k-1)}{h^{\uparrow}(\ell)}
\qquad\text{for}\qquad k \ge 1.
\end{equation} 
Otherwise  $ \mathcal{A}( \overline{ \mathfrak{e}}_{i})$ is identified with another edge on the boundary, splitting the hole into two parts, only one contains the infinite part of the map. The probability that the hole created on the left of the peeled edge is finite and has perimeter $2k$ with $k \geq 0$ is
\begin{equation}\label{eq:qp2}
p_{-k}^{(\ell)} \coloneqq \frac{1}{2}\nu(-k-1) \frac{h^\uparrow(\ell-k-1)}{h^{\uparrow}(\ell)}
\qquad\text{for}\qquad 0 \le k \le \ell-2,
\end{equation}
and similarly when ``left'' is replaced by ``right''. On these events, the finite holes created are filled-in with an independent map of law $ \mathbb{P}^{(k)}$.  Notice that $\sum_{k=1}^\infty p_k^{(\ell)} + 2\sum_{k=0}^{\ell-2} p_{-k}^{(\ell)}=1$ is ensured precisely because $h^\uparrow$ is $\nu$-harmonic. Reformulating the above transitions we have:

\begin{lem}[\cite{Budd:The_peeling_process_of_infinite_Boltzmann_planar_maps}]\label{lem:budd15}
Let $(\overline{\mathfrak{e}}_i)_{i \ge 0}$ be a peeling exploration of $\Map_\infty$ and for every $i \ge 0$, let $P_i = \frac{1}{2} |\partial \overline{\mathfrak{e}}_i|$ be the half-perimeter of $\overline{\mathfrak{e}}_i$ and $V_i = |\overline{\mathfrak{e}}_i|$ be its volume. Then $(P_i, V_i)_{i \ge 0}$ is a Markov chain whose law does not depend on the peeling algorithm $\mathcal{A}$. More precisely,
\begin{itemize}
\item $(P_i)_{i \ge 0}$ has the same law as $S^\uparrow$ the random walk started from $1$ and with i.i.d.\, increments of law $\nu$ given in \eqref{eq:def_nu} conditioned to never hit $\Z_-$.
\item Conditional on $(P_i)_{i \ge 0}$, the random variables $(V_{i+1} - V_i)_{i \ge 0}$ are independent, each $V_{i+1} - V_i$ is null if $P_i - P_{i+1}-1 \le 0$, otherwise it is distributed as the volume of a map sampled from $\P^{(\ell)}$ defined in \eqref{eq:def_free_Boltzmann_law_with_boundary}, where $\ell = P_i - P_{i+1}-1$. 
\end{itemize}
\end{lem}

Let us point out that $P_0 = 1$ comes from our convention that every rooted bipartite map can be seen as a planar map with a root-face of degree $2$.

\subsection{\texorpdfstring{\emph{Intermezzo}: The map $\Map^{(\infty)}$, a half-plane version of $\Map_\infty$}%
	{Intermezzo: The map M\^{}(infinity), a half-plane version of M\_infinity}}
\label{sec:half_plane_map}

Let us next briefly introduce another model of infinite random planar maps which we shall consider in Sections \ref{sec:lower_bound_height} and \ref{sec:perco}. We refer to \cite[Chapter 4.1]{Curien:Peccot} for details. The laws $\P^{(\ell)}$ converge weakly for the local topology as $\ell \to \infty$; we shall denote by $\P^{(\infty)}$ the limit, which is now a distribution on the set of one-ended rooted bipartite maps with a root-face of infinite degree. These are commonly referred to a maps of the half-plane. We may consider a peeling process $(\overline{\mathfrak{e}}_i)_{i \ge 0}$ on such a map, and analogs of \eqref{eq:qp} and \eqref{eq:qp2} hold if the perimeter of the hole is understood here in terms of algebraic variation of the number of edges with respect to the initial state (so it can be negative).

At each step $i \ge 0$, conditional on the current exploration $\overline{\mathfrak{e}}_i$, once an edge is selected, its peeling 
\begin{enumerate}
\item either leads to a new face of degree $2k$ with $k \ge 1$ with probability $\nu(k-1)$;
\item or the edge is identified with another edge on the boundary to its right or to its left, thus creating a new finite hole of perimeter $2k$ with probability $\nu(-k-1)/2$ where $k \geq 0$. Conditionally on this event the finite hole is filled-in with an independent map of law $ \mathbb{P}^{(k)}$.
\end{enumerate}

Recall Lemma \ref{lem:budd15}, the variation of the half-perimeter process associated with any peeling exploration of a random map $\Map^{(\infty)}$ sampled from $\P^{(\infty)}$ has now simply the law of $S$, the unconditioned random walk started from $0$ and with i.i.d. increments of law $\nu$ given in \eqref{eq:def_nu}.

\subsection{Scaling limits for the perimeter and volume process}
We now give the scaling limit of the perimeter and volume process in a peeling exploration of $ \Map_{\infty}$. The result and its proof are similar to \cite[Theorem 3.6]{Budd-Curien:Geometry_of_infinite_planar_maps_with_high_degrees} although we need additional ingredients to tackle the particular case of the Cauchy process. Let us first introduce the limiting processes. \medskip

Let $\Upsilon = (\Upsilon_t ; t \ge 0)$ be \emph{the symmetric Cauchy process}: $\Upsilon$ is a Lévy process with no drift, no Brownian part and with Lévy measure normalised to $\Pi(\d x) = |x|^{-2} \ind{x \ne 0} \d x$, so that $\E[\ex^{\i \lambda \Upsilon_t}] = \ex^{-t \pi |\lambda|}$ for every $t > 0$ and $\lambda \in \R$. Just as the random walk $S^\uparrow$ conditioned to never hit $\Z_-$, we may define $\Upsilon^\uparrow$ to be a version of $\Upsilon$ conditioned to remain positive via a Doob $h$-transform, using the harmonic function $h : x \mapsto \sqrt{x}$, see e.g. Caravenna \& Chaumont \cite[Section 1.2]{Caravenna-Chaumont:Invariance_principles_for_random_walks_conditioned_to_stay_positive}.

To construct the scaling limit of the volume process we proceed as in  \cite[Section 3.3]{Budd-Curien:Geometry_of_infinite_planar_maps_with_high_degrees}: We first consider $\xi_\bullet$ a positive $2/3$-stable random variable with Laplace transform
\[\Es{\ex^{- \lambda \xi_\bullet}} = \exp\left(- \big(\Gamma(5/2) \lambda \big)^{2/3}\right).\]
Since $\E[\xi_\bullet^{-1}] = \int_0^\infty \Gamma(5/2)^{-1} \exp(-x^{2/3}) \d x = 1$, one can define a random variable $\xi$, with mean $\E[\xi]=1$, by
\[\Es{f(\xi)} = \Es{\xi_\bullet^{-1} f(\xi_\bullet)},\]
for every non-negative and measurable function $f$. Let $(\xi^{(i)})_{i \ge 1}$ be a sequence of independent random variables distributed as $\xi$ and let $\chi = (\chi_t)_{t \ge 0}$ be a \cadlag{} process independent of the sequence $(\xi^{(i)})_{i \ge 1}$. We define then another process $\mathcal{V}(\chi) = (\mathcal{V}(\chi)(t))_{t \ge 0}$ by
\[\mathcal{V}(\chi)(t) = \sum_{t_i \le t} \xi^{(i)} \cdot |\Delta \chi(t_i)|^{3/2} \cdot \ind{\Delta \chi(t_i) < 0},\]
where $t_{1}, t_{2}, \ldots$ is a measurable enumeration of the jump times of $ \chi$. In general, the process $\mathcal{V}(\chi)$ above may be infinite, but since $x \mapsto x^{3/2} \ind{x <0}$ integrates the L\'evy measure of $\Upsilon$ in the neighbourhood of $0$ it is easy to check that $\mathcal{V}(\Upsilon)$ is a.s. finite. It can be shown using absolute continuity relations between $\Upsilon$ and $\Upsilon^\uparrow$ that the process $\mathcal{V}(\Upsilon^\uparrow)$ is also almost surely finite.

\begin{thm}[General peeling growth]\label{thm:convergence_processes_perimeter_volume}
Let $(P_i, V_i)_{i \ge 0}$ be respectively the half-perimeter and the number of inner vertices in a peeling exploration of $\Map_\infty$. We have the following convergence in distribution in the sense of Skorokhod
\[\left(n^{-1} P_{[nt]}, n^{-3/2} V_{[nt]}\right)_{t \ge 0}
\cvdist
\left(\mathsf{p}_\q \cdot \Upsilon^\uparrow(t), \mathsf{v}_\q \cdot \mathcal{V}(\Upsilon^\uparrow)(t)\right)_{t \ge 0},\]
where $\mathsf{v}_\q = \frac{2}{c_\q} (\frac{\mathsf{p}_\q}{\pi})^{1/2}$.
\end{thm}

\begin{figure}[!ht]
\begin{center}
\includegraphics[width=.6\linewidth]{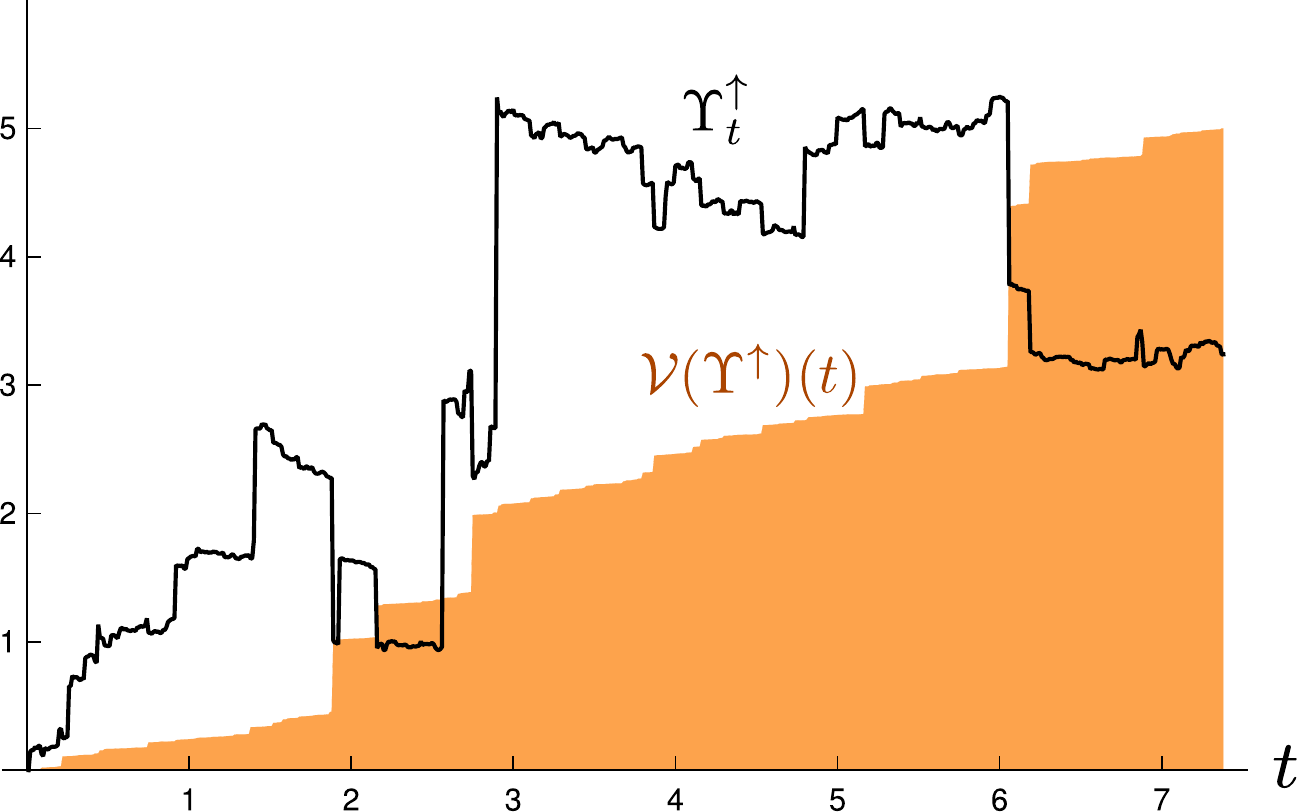}
\caption{Simulation of the processes $\Upsilon^\uparrow$ and $\mathcal{V}(\Upsilon^\uparrow)$.}
\end{center}
\end{figure}

\begin{proof}
Recall that $S$ is a random walk with i.i.d. steps sampled from $\nu$ given in \eqref{eq:def_nu} and that $P$ has the law of $S^\uparrow$. A result similar to our claim for $a \ne 2$ was proved in \cite[Section 3.3]{Budd-Curien:Geometry_of_infinite_planar_maps_with_high_degrees}; the arguments there are based on the convergence in distribution of $(n^{-1/(a-1)} S_{[nt]})_{t \ge 0}$ towards a certain $(a-1)$-stable Lévy process\footnote{Let us point out that the convergence of $S^\uparrow$ given that of $S$ is due to Caravenna \& Chaumont \cite{Caravenna-Chaumont:Invariance_principles_for_random_walks_conditioned_to_stay_positive}.} and they all carry through without modification for $a=2$. Therefore, we only need to prove the corresponding invariance principle for $S$ in the case $a=2$, which is the claim of Proposition \ref{prop:convergence_walk_Cauchy} below.
\end{proof}

\begin{rem}\label{rem:general_peeling_vertices_faces}
Echoing Remark 3.5 in \cite{Budd-Curien:Geometry_of_infinite_planar_maps_with_high_degrees}, we could have chosen the number of faces in a peeling exploration of $\Map_\infty$ as the notion of volume; the only effect would have been to replace the constant $\mathsf{v}_\q$ by $\mathsf{v}_\q' = \frac{1}{2} (1 - \frac{4}{c_\q}) (\frac{\mathsf{p}_\q}{\pi})^{1/2}$.
\end{rem}

Recall that $\Upsilon$ is a symmetric Cauchy process and $S$ is a random walk with increments of law $\nu$.

\begin{prop}\label{prop:convergence_walk_Cauchy}
The following convergence in distribution holds for the Skorokhod topology:
\[\left(n^{-1} S_{[nt]}\right)_{t \ge 0} \cvdist (\mathsf{p}_\q \cdot \Upsilon_t)_{t \geq 0}.\]
\end{prop}

As the proof will show, it is easy to see that for some centring sequence $(b_n)_{n \ge 1}$, we have the convergence in distribution $(n^{-1} S_{[nt]} - b_n t)_{t \ge 0} \to (\mathsf{p}_\q \Upsilon_t)_{t \geq 0}$. The main point is to prove that $b_n$ can be set to $0$; this was rather simple in \cite{Budd-Curien:Geometry_of_infinite_planar_maps_with_high_degrees} in the case $a \ne 2$ but is more involved here. The key idea is to use the fact that the explicit function $ h^{\uparrow}$ is a harmonic function for the walk $S$ killed when entering $\Z_{-}$.

\begin{proof}
By a classical result on random walks (see e.g. Jacod
\& Shiryaev \cite[Chapter VII]{Jacod-Shiryaev:Limit_theorems_for_stochastic_processes}), it is equivalent to show the convergence in distribution
\[n^{-1} S_n \cvdist \mathsf{p}_\q \Upsilon_1.\]
First, it is easy to see from its tail behaviour \eqref{eq:tail_nu} that $\nu$ belongs to the domain of attraction of a symmetric stable law with index one, i.e. there exists two sequences $(a_n)_{n \ge 1}$ and $(b_n)_{n \ge 1}$ such that $a_n^{-1} S_n - b_n$ converges towards some symmetric Cauchy distribution $C_1$, see e.g. \cite[Theorem 8.3.1]{Bingham-Goldie-Teugels:Regular_variation}. As above, it follows that the convergence in distribution
\begin{equation}\label{eq:cv_walk_cauchy}
\left(a_n^{-1} S_{[nt]} - b_n t\right)_{t \ge 0} \cvdist (C_t)_{t \geq 0}
\end{equation}
holds for the Skorokhod topology, where $C = (C_t)_{t \geq 0}$ is some symmetric Cauchy process. Furthermore, one sees from the tails of $\nu$ that one can take $a_n=n$ in \eqref{eq:cv_walk_cauchy}; we aim at showing that $b_n$ can be set to $0$. A straightforward calculation using the jump measure then shows that $C$ has the same law as $\mathsf{p}_\q \Upsilon$.

In the remainder of this proof, we assume that \eqref{eq:cv_walk_cauchy} holds with $a_n=n$ and we show that $b_n \to 0$ as $n \to \infty$. According to \cite[Theorem 8.3.1]{Bingham-Goldie-Teugels:Regular_variation}, one may take $b_n = \Es{\frac{S_1}{1+(S_1/n)^2}}$ but we cannot perform bare-hand calculations and so we use the explicit harmonic function $h^\uparrow$. We shall prove by contradiction that there is no sequence of integers along which $b_n$ converges to an element of $[-\infty, 0[ \cup ]0, \infty]$; we shall treat the three cases $\infty$, $-\infty$ and $]-\infty, 0[ \cup ]0, \infty[$ separately and drop the mention ``along a sequence of integers'' to simplify the notation. 

We let $S_0=n$ and define a stopping time $\sigma = \inf\{k \ge 1 : S_k \notin [n/2, 2n]\}$. We denote by $\P_n$ the law of $S$ started from $n$ and similarly by $\P_1$ the law of $\Upsilon$ started from $1$. Since $h^\uparrow$ is $\nu$-harmonic on $\N$, we have
\begin{equation}\label{eq:stopping_thm}
\E_n\left[h^\uparrow(S_{n \wedge \sigma})\right] = h^\uparrow(n).
\end{equation}

Let us first suppose that $b_n \to \infty$, then, by \eqref{eq:cv_walk_cauchy}, we have $\P_n(S_{n \wedge \sigma} > 2n) \to 1$. Observe that $h^\uparrow$ is non-decreasing and that $h^\uparrow(n) \sim 2 \sqrt{n/\pi}$ as $n \to \infty$; appealing to \eqref{eq:stopping_thm}, it follows that
\[h^\uparrow(n)
\ge \E_n\left[h^\uparrow(S_{n \wedge \sigma}) \ind{S_{n \wedge \sigma} > 2n}\right]
\ge h^\uparrow(2n) (1+o(1)),\]
which leads to a contradiction when letting $n \to \infty$.

Suppose next that $b_n \to -\infty$, then, by \eqref{eq:cv_walk_cauchy} again, we have $\P_n(S_{n \wedge \sigma} < n/2) \to 1$. We write \eqref{eq:stopping_thm} as
\[h^\uparrow(n)
= \E_n\left[h^\uparrow(S_{n \wedge \sigma}) \ind{S_{n \wedge \sigma} < n/2}\right] + \E_n\left[h^\uparrow(S_{n \wedge \sigma}) \ind{S_{n \wedge \sigma} \ge n/2}\right].\]
The first term on the right-hand side is bounded above by $h^\uparrow(n/2)$; we show that the second term is small compared to $\sqrt{n}$ and conclude again a contradiction. Let $\Delta_n = \max(0, S_1-S_0, \dots, S_n-S_{n-1})$ be the largest jump of $S$ up to time $n$, clearly, $S_{n \wedge \sigma} \le 2n+\Delta_n$. From the tail \eqref{eq:tail_nu} of $\nu$, there exists a constant $K > 0$ such that $\P_n(n^{-1} \Delta_n \ge x) \le n \cdot \nu([nx, \infty)) \le K x^{-1}$ for all $n \ge 1$ and $x > 0$. Using the bound $h^\uparrow(k) \le 2\sqrt{k}$ for every $k \ge 1$, this shows that for every $q > 0$, there exists constants $K(q)$ which may vary from line to line such that
\begin{align*}
\E_n\left[\left(n^{-1/2} h^\uparrow(S_{n \wedge \sigma})\right)^q\right]
&\le \E_n\left[2^q \left(2+n^{-1} \Delta_n\right)^{q/2}\right]
\\
&\le K(q) \E_n\left[\left(n^{-1} \Delta_n\right)^{q/2}\right]
\\
&\le K(q) \int_1^\infty \P_n(n^{-1} \Delta_n \ge x^{2/q}) \d x
\\
&\le K(q) \int_1^\infty K x^{-2/q} \d x,
\end{align*}
which is finite whenever $q<2$. In particular, the sequence $(n^{-1/2} h^\uparrow(S_{n \wedge \sigma}))_{n \ge 1}$ is uniformly integrable; since $\P_n(S_{n \wedge \sigma} \ge n/2) \to 0$, we indeed conclude that $n^{-1/2} \E_n[h^\uparrow(S_{n \wedge \sigma}) \ind{S_{n \wedge \sigma} \ge n/2}] \to 0$.

Finally, let us assume that $b_n \to b$ as $n\to \infty$, with $b \in ]-\infty, \infty[$. Then \eqref{eq:cv_walk_cauchy} with $a_n=n$ yields
\[\left(n^{-1} S_{[nt]}\right)_{t \ge 0} \cvdist (C^b_t)_{t \geq 0},\]
for the Skorokhod topology, where $C^b = (C_t+bt)_{t \geq 0}$ is a symmetric Cauchy process with a drift $b$. We aim at showing that $g : x \mapsto \sqrt{x} \cdot \ind{x>0}$ is harmonic for $C^b$ on $(0, \infty)$; recall that it is also harmonic for $C$, we thus obtain
\[\E_1\left[\sqrt{C_1} \cdot \ind{\inf_{0 \le t \le 1} C_t > 0}\right] = 1 = \E_1\left[\sqrt{C_1+b} \cdot \ind{\inf_{0 \le t \le 1} C_t+bt > 0}\right],\]
and so $b=0$. For every $x \ge 1$, set
\[\Theta(x) = \inf\{t > 0 : C^b_t \notin (0,x)\}
\qquad\text{and}\qquad
\theta(x) = \inf\{k \ge 1 : k^{-1} S_k \notin (0,x)\}.\]
Fix a (large) $x$ and observe that $n^{-1/2} h^\uparrow(S_{n \wedge \theta(x)})$ under $\P_n$ converges in distribution as $n \to \infty$ towards $2 \pi^{-1/2} g(C^b_{1 \wedge \Theta(x)})$ under $\P_1$. Moreover, as above, we have $n^{-1/2} h^\uparrow(S_{n \wedge \theta(x)}) \le n^{-1/2} h^\uparrow(nx + \Delta_n) \le 2 (x + n^{-1} \Delta_n)^{1/2}$, and the latter sequence is uniformly integrable. It follows that
\[\E_1\left[g(C^b_{1 \wedge \Theta(x)})\right] 
= \frac{\sqrt{\pi}}{2} \lim_{n \to \infty} \E_n\left[n^{-1/2} h^\uparrow(S_{n \wedge \theta(x)})\right] 
= \frac{\sqrt{\pi}}{2} \lim_{n \to \infty} n^{-1/2} h^\uparrow(n) 
= 1.\]
In only remains to show that the left-hand side converges to $\E_1[(C^b_1)^{1/2}\ind{\inf_{0 \le t \le 1} C^b_t > 0}]$ as $x \to \infty$. Again, it suffices to show that $((C^b_s)^{1/2}\ind{\inf_{0 \le t \le 1} C^b_t > 0} ; 0 \le s \le 1)$ is uniformly integrable. This follows from the easy bound $(C^b_s)^{1/2} \ind{\inf_{0 \le t \le 1} C^b_t > 0} \le |b|^{1/2} + (\sup_{0 \le u \le 1} C_u)^{1/2}$ for every $s \in [0,1]$ and the tail of the random variable $\sup_{0 \le u \le 1} C_u$, which can be found in \cite{Bertoin:Levy_processes} Proposition 4, page 221, which shows that all moments smaller than $2$ are finite (let us mention that Darling \cite{Darling:The_maximum_of_sums_of_stable_random_variables} expresses at the very end of the paper the density of $\sup_{0 \le u \le 1} C_u$ when $C$ is the standard Cauchy process, i.e. $\pi^{-1} \Upsilon$).
\end{proof}

\subsection{More on the perimeter and volume process}

In view of our coming proofs, we shall need a few more results on the perimeter and volume process. First, we recall the following result  \cite[Lemma 5.8]{Budd-Curien:Geometry_of_infinite_planar_maps_with_high_degrees} which is stated for $a \in (3/2, 2)$ there, but the arguments still hold in the case $a=2$. 

\begin{lem}[\cite{Budd-Curien:Geometry_of_infinite_planar_maps_with_high_degrees}]
\label{lem:as_convergence_perimeter_volume}
The following almost sure convergences hold:
\[\frac{\log P_n}{\log n} \cvas 1,
\qquad\text{and}\qquad
\frac{\log V_n}{\log n} \cvas \frac{3}{2}.\]
\end{lem}

Observe that the fact such convergences hold in probability is a direct consequence of Theorem \ref{thm:convergence_processes_perimeter_volume}; nonetheless,  in order to prove Theorems \ref{thm:intro_dual} and \ref{thm:intro_Eden}, we shall evaluate the process $V$ at a certain random time tending to infinity, therefore an almost sure convergence as above is required, see Sections \ref{sec:fpp} and \ref{sec:dual}.

We continue with a conditional local limit theorem. We denote by $f^\uparrow$ the density of $\Upsilon^\uparrow_1$. For each $x \in \N$, we let $\P_x$ be the law of $S$ or $S^\uparrow$ started from $x$.

\begin{lem}\label{lem:LLT_perimeter}
Consider the asymptotic relation as $n \to \infty$
\begin{equation}\label{eq:conditional_LLT}
\P_x(S^\uparrow_n=y) = \frac{1}{n \mathsf{p}_\q} \cdot \frac{h^\uparrow(y) \sqrt{\pi}}{2\sqrt{y}} \cdot \left(f^\uparrow\left(\frac{y}{\mathsf{p}_\q n}\right) + \sqrt{\frac{y}{n}} o(1)\right),
\end{equation}
with $x, y \in \N$.
\begin{enumerate}[ref={\thelem(\roman*)}]
\item\label{lem:LLT_perimeter_1} If $x=1$, then \eqref{eq:conditional_LLT} holds uniformly in $y \ge 1$.

\item\label{lem:LLT_perimeter_2} For every $\varepsilon \in (0,1)$ and every sequence $\delta_n \to 0$, \eqref{eq:conditional_LLT} holds uniformly for $x \in [1, n \delta_n]$ and $y \in [\varepsilon n, \varepsilon^{-1} n]$.
\end{enumerate}
\end{lem}

Let us note that the ratio $(h^\uparrow(y) \sqrt{\pi})/(2\sqrt{y})$ of the harmonic functions associated with $S$ and $\Upsilon$ tends to $1$ as $y \to \infty$.

\begin{proof}
Let us set $c_n = \mathsf{p}_\q n$ so $c_n^{-1} S_n$ converges to $\Upsilon_1$ in distribution from Proposition \ref{prop:convergence_walk_Cauchy}. Set $\tau^- = \inf\{k \ge 1 : S_k \le 0\}$, then
\[\P_x(S^\uparrow_n=y) = \frac{h^\uparrow(y)}{h^\uparrow(x)} \cdot \P_x(S_n=y \text{ and } \tau^- > n)\]
for all $n, x, y \ge 1$. Let us consider the meander $\tilde{\Upsilon}$ which is informally the process $\Upsilon$ conditioned to remain positive \emph{up to time $1$}; we denote by $\tilde{f}$ the density of $\tilde{\Upsilon}_1$. On the one hand, according to Vatutin \& Wachtel \cite[Theorem 5]{Vatutin-Wachtel:Local_probabilities_for_random_walks_conditioned_to_stay_positive}, uniformly in $y \ge 1$,
\[\P_1(S_n=y \text{ and } \tau^- > n) = \frac{\Pr{\tau^- > n}}{c_n} \cdot \left(\tilde{f}(y/c_n) + o(1)\right),\]
as $n \to \infty$. On the other hand, according to Doney \cite[Proposition 11]{Doney:Local_behaviour_of_first_passage_probabilities}, for every $\varepsilon \in (0,1)$ and every sequence $\delta_n \to 0$, uniformly for $x \in [1, n \delta_n]$ and $y \in [n \varepsilon, n/\varepsilon]$,
\[\P_x(S_n=y \text{ and } \tau^- > n)
= h^\uparrow(x) \cdot \frac{\Pr{\tau^- > n}}{c_n} \cdot \tilde{f}(y/c_n) \cdot (1 + o(1)),\]
as $n \to \infty$. Moreover, according to the discussion on page 181 of \cite{Vatutin-Wachtel:Local_probabilities_for_random_walks_conditioned_to_stay_positive}, since $\P(S_n > 0) \to 1/2$, there exists a function $\ell$ slowly varying at infinity such that, as $n \to \infty$,
\[\Pr{\tau^- > n} \sim  \frac{\ell(n)}{\sqrt{n}},
\qquad\text{and}\qquad
\Pr{\tau^+ > n} \sim  \frac{1}{\pi \ell(n) \sqrt{n}},\]
where $\tau^+ = \inf\{k \ge 1 : S_k > 1\}$. Moreover, Equation 31 in \cite{Vatutin-Wachtel:Local_probabilities_for_random_walks_conditioned_to_stay_positive} applied to $1-S$ reads
\[h^\uparrow(n) \sim C \cdot n \cdot \Pr{\tau^+ > n}
\qquad\text{as}\qquad n \to \infty,\]
for some constant $C > 0$. Since we know that $h^\uparrow(n) \sim 2\sqrt{n/\pi}$, we conclude that, as $n \to \infty$,
\[\Pr{\tau^- > n} \sim \frac{1}{\pi \cdot n \cdot \P(\tau^+ > n)}
\sim \frac{C}{2\sqrt{\pi n}}.\]
This shows that there exists a constant $K > 0$ such that the asymptotic relation as $n \to \infty$
\begin{equation}\label{eq:conditional_LLT_meander_density}
\P_x(S^\uparrow_n=y) = \frac{K}{\mathsf{p}_\q} \cdot h^\uparrow(y) \cdot n^{-3/2} \cdot \left(\tilde{f}\left(\frac{y}{\mathsf{p}_\q n}\right) + o(1)\right),
\end{equation}
with $x, y \in \N$ holds uniformly in our two regimes.

According to Doney \& Savov \cite[Remark 12]{Doney-Savov:The_asymptotic_behavior_of_densities_related_to_the_supremum_of_a_stable_process}, there exists a constant $C > 0$ such that the density $f^\uparrow$ of $\Upsilon^\uparrow_1$ is given by $f^\uparrow(x) = C x^{1/2} \tilde{f}(x)$. Therefore, \eqref{eq:conditional_LLT_meander_density} reads
\[\P_x(S^\uparrow_n=y) 
= \frac{1}{C\sqrt{\mathsf{p}_\q}} \cdot \frac{h^\uparrow(y)}{\sqrt{y}} \cdot \frac{1}{n} \cdot \left(f^\uparrow\left(\frac{y}{\mathsf{p}_\q n}\right) + \sqrt{\frac{y}{n}} o(1)\right).\]
By integrating this identity for $x=1$ and using that $h^\uparrow(y)/\sqrt{y} \to 2/\sqrt{\pi}$, we easily obtain for every $0 < a < b < \infty$,
\[\Pr{S^\uparrow_n \in (an, bn)} \cv 2C \sqrt{\frac{\mathsf{p}_\q}{\pi}} \Pr{\mathsf{p}_\q \Upsilon^\uparrow_1 \in (a,b)}.\]
We conclude from Theorem \ref{thm:convergence_processes_perimeter_volume} that $C^{-1} = 2 \sqrt{\mathsf{p}_\q/\pi}$, which completes the proof.
\end{proof}

Lemma \ref{lem:LLT_perimeter_2} and the bound $h^\uparrow(y) \le 2 \sqrt{y}$ valid for all $y \ge 1$ enable us to compare the distribution of $P_n$ when started from two different values at time $0$.

\begin{cor}\label{cor:bound_conditional_LLT}
For every $\varepsilon \in (0,1)$ and every sequence $\delta_n \to 0$, uniformly for $x, x' \in [1, n \delta_n]$ and $y \in [\varepsilon n, \varepsilon^{-1} n]$,
\[\left|\P_x(P_n=y) - \P_{x'}(P_n=y)\right| = \frac{o(1)}{n \sqrt{\varepsilon}},
\qquad\text{as}\qquad n \to \infty.\]
\end{cor}

Appealing to Lemma \ref{lem:LLT_perimeter_1}, we next show that the sequence $(n/P_n)_{n \ge 1}$ is uniformly integrable.

\begin{lem}\label{lem:bound_moments_perimeter_FPP}
For every $0 < q < 3/2$, we have
\[\sup_{n \ge 1} \Es{\left(\frac{n}{P_n}\right)^q} < \infty.\]
\end{lem}

\begin{proof}
Fix $q > 0$ and $n \ge 1$; Lemma \ref{lem:LLT_perimeter_1} and the bound $h^\uparrow(k) \le 2 \sqrt{k}$ for all $k \ge 1$ yield
\begin{align*}
\Es{\left(\frac{n}{P_n}\right)^q}
&\le 1 + \Es{\left(\frac{n}{P_n}\right)^q \ind{P_n \le n}}
\\
&= 1 + \sum_{k=1}^n \left(\frac{n}{k}\right)^q \cdot \frac{1}{n \mathsf{p}_\q} \cdot \frac{h^\uparrow(k) \sqrt{\pi}}{2\sqrt{k}} \cdot \left(f^\uparrow\left(\frac{k}{\mathsf{p}_\q n}\right) + \sqrt{\frac{k}{n}} o(1)\right)
\\
&\le 1 + \frac{\sqrt{\pi}}{n \mathsf{p}_\q} \sum_{k=1}^n \left(\frac{n}{k}\right)^q \cdot \left(f^\uparrow\left(\frac{k}{\mathsf{p}_\q n}\right) + \sqrt{\frac{k}{n}} o(1)\right)
\\
&= 1 + \frac{\sqrt{\pi}}{\mathsf{p}_\q} \left(\int_0^1 \frac{f^\uparrow(x/\mathsf{p}_\q) \d x}{x^q} (1+o(1)) + \int_0^1 \frac{\d x}{x^{q-1/2}} o(1)\right).
\end{align*}
According to Doney \& Savov \cite[Remark 12]{Doney-Savov:The_asymptotic_behavior_of_densities_related_to_the_supremum_of_a_stable_process}, there exists a constant $c>0$ such that $f^\uparrow(x) \sim cx$ as $x \downarrow 0$ so both integrals above converge whenever $0 < q < 3/2$ and the claim follows.
\end{proof}

\begin{rem}\label{rem:moment_inverse_Cauchy_positive}
We know from Theorem \ref{thm:convergence_processes_perimeter_volume} that $n/P_n$ converges in distribution to 
$(\mathsf{p}_\q \Upsilon^\uparrow_1)^{-1}$ as $n \to \infty$. Since Lemma \ref{lem:bound_moments_perimeter_FPP} implies that $(n/P_n)_{n \ge 1}$ is uniformly integrable, we conclude that
\[\Es{n/P_n} \cv \Es{(\mathsf{p}_\q \Upsilon^\uparrow_1)^{-1}}.\]
It turns out that the value of the limit is explicit:
\begin{equation}\label{eq:moment_inverse_Cauchy_positive}
\Es{(\mathsf{p}_\q \Upsilon^\uparrow_1)^{-1}} = \frac{2}{\pi^2 \mathsf{p}_\q}.
\end{equation}
Indeed, the process $\Upsilon^\uparrow$ is a positive self-similar Markov process; let us denote by $\xi^\uparrow$ the associated Lévy process in the Lamperti representation, then, according to Bertoin \& Yor \cite{Bertoin-Yor:The_entrance_laws_of_self_similar_Markov_processes_and_exponential_functionals_of_Levy_processes}, we have the identity $\E[(\Upsilon^\uparrow_1)^{-1}] = \E[\xi^\uparrow_1]^{-1}$. Kyprianou, Pardo \& Rivero \cite[Proposition 2]{Kyprianou-Pardo-Rivero:Exact_and_asymptotic_n_tuple_laws_at_first_and_last_passage} express the characteristic exponent of $\xi^\uparrow$, from which we find $\E[\xi^\uparrow_1] = \pi^2/2$.
\end{rem}

\section{First passage percolation distance}
In this section we prove Theorem \ref{thm:intro_Eden}. The technique of proof is similar to that of  \cite[Theorem 4]{Curien-Le_Gall:Scaling_limits_for_the_peeling_process_on_random_maps} or \cite[Proposition 4.1]{Budd-Curien:Geometry_of_infinite_planar_maps_with_high_degrees}, but again the characteristics of the Cauchy-type walks driving the perimeter process lead to new phenomena that were absent in previous works and which require additional arguments. \medskip 
\label{sec:fpp}

Let $\map$ be an infinite one-ended bipartite map. We consider its dual map $ \map^{\dagger}$ and equip independently each edge $e$ of $\map^{\dagger}$ with a random weight $x_e$ distributed according to the exponential law of mean $1$, i.e. with density $\ex^{-x} \d x \ind{x >0}$. We define then the fpp-distance\footnote{This model on the dual map $\map^\dagger$ is often referred to as the \emph{Eden model} on $\map$ \cite{Ambjorn-Budd:Multi_point_functions_of_weighted_cubic_maps}.} on $ \map^{\dagger}$ which modifies the usual dual graph metric on $\map^{\dagger}$: for every pair $u, v$ of vertices of $\map^{\dagger}$ (i.e.~faces of $\map$), we set
\[\mathrm{d_{fpp}}(u, v) = \inf \sum_{e \in \gamma} x_{e},\]
where the infimum is taken over all paths $\gamma : u \to v$ in $\map^{\dagger}$. For any $r \geq 0$, we denote by $ \mathrm{Ball}_{r}^{ \mathrm{Eden}}( \map)$ the set of faces of $\map$ which are within fpp-distance less than $r$ from the root-face of $\map$. Then as usual, we consider its hull $\hBallEden( \map)$, obtained by filling-in all the finite components of its complement: it is then a sub-map of $\map$ with a single hole. Recall that $| \hBallEden( \map)|$ is  the number of inner vertices of the map and  $|\partial \hBallEden( \map)|$ the perimeter of its hole. Our main result is the following, which readily implies and extends Theorem \ref{thm:intro_Eden} (after translation to the dual version).

\begin{thm}[First passage percolation growth]\label{thm:perimeter_volume_fpp}
The following convergences in probability hold:
\[\frac{\log |\partial \hBallEden(\Map_\infty)|}{r} \cvproba[r] \pi^2 \mathsf{p}_\q
\qquad\text{and}\qquad
\frac{\log |\hBallEden( \Map_\infty)|}{r} \cvproba[r] \frac{3}{2} \pi^2 \mathsf{p}_\q.\]
\end{thm}

The proof uses the peeling exploration described in Section \ref{sec:def_peeling_general} with a random algorithm that we now describe, called the \emph{uniform peeling}. We refer to \cite[Section 2.4]{Budd-Curien:Geometry_of_infinite_planar_maps_with_high_degrees} for details. Given a bipartite map $\map$, we construct a sequence $(\overline{\mathfrak{e}}_i)_{i \ge 0}$ as in Section \ref{sec:def_peeling_general} by selecting for each $i \ge 0$, conditional on $\overline{\mathfrak{e}}_i$, an edge $\mathcal{A}(\overline{\mathfrak{e}}_i)$ uniformly at random on the boundary of $\overline{\mathfrak{e}}_i$. On the other hand, observing an Eden model on $\map$, it is easy to see that the process $(\hBallEden[t]( \map))_{t \ge 0}$ admits jump times $0 = t_0 < t_1 < \cdots$ and that the sequence  $(\hBallEden[t_i]( \map))_{i \geq 0}$ is a peeling exploration of $\map$. Based on the properties of exponential distribution, \cite[Proposition 2.3]{Budd-Curien:Geometry_of_infinite_planar_maps_with_high_degrees} shows:
\begin{enumerate}
\item\label{item:Eden_peeling_unif} the sequence $(\hBallEden[t_i]( \map))_{i \ge 0}$ has the same law as a uniform peeling of $\map$;
\item\label{item:Eden_peeling_jump_times} conditional on $(\hBallEden[t_i]( \map))_{i \ge 0}$, the random variables $t_{i+1}-t_i$ are independent and distributed respectively according to the exponential law of parameter $|\partial \hBallEden[t_i](\map)|$.
\end{enumerate}

Let us write $(T_i)_{i \ge 0}$ for the sequence of jump times of the process $(\hBallEden[t](\Map_\infty))_{t \ge 0}$ and for each $i \ge 0$, let $P_i =  \frac{1}{2}|\partial \hBallEden[T_i]( \Map_{\infty})|$ and $V_i =  |\hBallEden[T_i]( \Map_{\infty})|$ be respectively the half-perimeter and volume associated with this uniform peeling. Then, according to Property \ref{item:Eden_peeling_jump_times} above, conditional on $(P_i)_{i \ge 0}$, we may write
\begin{equation}\label{eq:representation_jump_time_eden}
T_{n} = \sum_{i=0}^{n-1} \frac{ \mathbf{e}_{i}}{2P_{i}},
\qquad\text{for every}\qquad n \ge 1,
\end{equation}
where $(\mathbf{e}_{ i})_{i \geq 0}$ are independent exponential variables of expectation $1$. As discussed in Remark \ref{rem:moment_inverse_Cauchy_positive}, we have $\E[i/P_i] \to \E[(\mathsf{p}_\q \Upsilon^\uparrow_1)^{-1}] = 2/(\pi^2 \mathsf{p}_\q)$ as $i \to \infty$, so
  \begin{eqnarray} \label{eq:cvmean}\frac{1}{\log n} \sum_{i=1}^{n} \Es{\frac{1}{2 P_i}}
\cv \frac{1}{\pi^2 \mathsf{p}_\q}.  \end{eqnarray}
As we will see in Lemma \ref{lem:decorrelation_FPP}, it turns out that the process $(P_{i})_{i \geq 0}$ decorrelates over scales and so we can prove concentration of $\sum_{i=1}^{n} \frac{1}{2 P_i}$ around its mean:

\begin{prop}\label{prop:convergence_jump_times_FPP}
The following convergence holds:
\[\Es{\left| \frac{T_n}{\log n} - \frac{1}{\pi^2 \mathsf{p}_\q}\right|} \cv 0.\]
\end{prop}

Let us first prove Theorem \ref{thm:perimeter_volume_fpp} taking Proposition \ref{prop:convergence_jump_times_FPP} for granted.

\begin{proof}[Proof of Theorem \ref{thm:perimeter_volume_fpp}]
For every $t > 0$, let us set $U_t = \inf\{n \ge 1 : T_n > t\}$ and then write
\[\left(\frac{\log |\partial  \hBallEden[t]( \Map_{\infty})|}{t}, \frac{\log | \hBallEden[t]( \Map_{\infty})|}{t}\right)
= \frac{\log(U_t-1)}{t} \cdot \left(\frac{\log P_{U_t-1}}{\log(U_t-1)}, \frac{\log V_{U_t-1}}{\log(U_t-1)}\right).\]
Since $T_n$ is non-decreasing, Proposition \ref{prop:convergence_jump_times_FPP} implies that $t^{-1} \log U_t \to \pi^2 \mathsf{p}_\q$ in probability as $t \to \infty$, then Lemma \ref{lem:as_convergence_perimeter_volume} ensures that the pair in the right-hand side converges in probability to $(1, 3/2)$.
\end{proof}

We next turn to the proof of Proposition \ref{prop:convergence_jump_times_FPP}.

\begin{proof}[Proof of Proposition \ref{prop:convergence_jump_times_FPP}]
Recall that $(P_n)_{n \ge 0}$ has the law of $(S^\uparrow_n)_{n \ge 0}$ started from $1$, where $S$ is the random walk with step distribution $\nu$ defined by \eqref{eq:def_nu}. We first claim that $(\log n)^{-1} T_n - (\log n)^{-1} \sum_{i=0}^{n-1} \frac{1}{2P_{i}}$ converges to $0$ in $\mathrm{L}^2$ as $n \to \infty$. Indeed, we have $P_i \ge 1$ for every $i \ge 0$ so, according to Lemma \ref{lem:bound_moments_perimeter_FPP}, there exists a constant $C > 0$ such that $\E[P_i^{-2}] \le \E[P_i^{-1}] \le C (1+i)^{-1}$ for all $i \ge 0$. By first conditioning with respect to $(P_k)_{k \ge 0}$ and then integrating, we obtain from \eqref{eq:representation_jump_time_eden},
\[4 \cdot \Es{\left(T_n - \sum_{i=0}^{n-1} \frac{1}{2P_{i}}\right)^2}
= \Es{\left(\sum_{i=0}^{n-1} \frac{\mathbf{e}_i-1}{P_{i}}\right)^2}
= \Es{\sum_{i=0}^{n-1} \frac{1}{P_{i}^2}}
\le \sum_{i=0}^{n-1} \frac{C}{i+1}
= o((\log n)^2).\]
Given the convergence in mean \eqref{eq:cvmean}, it is sufficient to show the convergence in $\mathrm{L}^1$:
\begin{equation}\label{eq:convergence_sum_inverse_perimeter}
\frac{1}{\log n} \sum_{i=1}^{n} \left(\frac{1}{P_{i}} - \Es{\frac{1}{P_{i}}}\right) 
\cvL 0.
\end{equation}

Fix $\varepsilon > 0$ and for every $i \ge 1$, set
\begin{equation}\label{eq:epsilon_cut_inverse_perimeter}
\begin{aligned}
X^{(\varepsilon)}_i &\coloneqq \frac{1}{P_i} \ind{P_i \in [\varepsilon i, \varepsilon^{-1} i]} - \Es{\frac{1}{P_i} \ind{P_i \in [\varepsilon i, \varepsilon^{-1} i]}}
\\
Y^{(\varepsilon)}_i &\coloneqq \frac{1}{P_i} \ind{P_i \notin [\varepsilon i, \varepsilon^{-1} i]} - \Es{\frac{1}{P_i} \ind{P_i \notin [\varepsilon i, \varepsilon^{-1} i]}},
\end{aligned}
\end{equation}
so $P_i^{-1} - \E[P_i^{-1}] = X^{(\varepsilon)}_i + Y^{(\varepsilon)}_i$. We shall prove that
\[\lim_{\varepsilon \downarrow 0} \limsup_{n \to \infty} 
\Es{\left(\frac{1}{\log n} \sum_{i=1}^n X^{(\varepsilon)}_i\right)^2}
= \lim_{\varepsilon \downarrow 0} \limsup_{n \to \infty} 
\Es{\left|\frac{1}{\log n} \sum_{i=1}^n Y^{(\varepsilon)}_i\right|}
= 0.\]
We handle easily the second term. Indeed, according to Lemma \ref{lem:bound_moments_perimeter_FPP}, for every $1 < q < 3/2$ and every $i$ large enough,
\begin{align*}
i \cdot \Es{\frac{1}{P_i} \ind{P_i \notin [\varepsilon i, \varepsilon^{-1} i]}}
&\le \Es{\left(\frac{i}{P_i}\right)^q}^{1/q} \cdot \Pr{P_i \notin [\varepsilon i, \varepsilon^{-1} i]}^{1-1/q}
\\
&\le C(q) \cdot\Pr{\mathsf{p}_\q \Upsilon^\uparrow_1 \notin [\varepsilon, \varepsilon^{-1}]}^{1-1/q},
\end{align*}
where $C(q)$ is a constant depending only on $q$.
Observe that for $q > 1$, the last expression tends to $0$ as $\varepsilon \downarrow 0$ and so
\[\lim_{\varepsilon \downarrow 0} \limsup_{n \to \infty} \frac{1}{\log n} \Es{\left|\sum_{i=1}^n Y^{(\varepsilon)}_i\right|}
\le \lim_{\varepsilon \downarrow 0} \limsup_{n \to \infty} \frac{2}{\log n} \sum_{i=1}^n \Es{\frac{1}{P_i} \ind{P_i \notin [\varepsilon i, \varepsilon^{-1} i]}}
= 0.\]
We next consider the $X^{(\varepsilon)}_i$'s. Fix $\varepsilon > 0$ and a non-decreasing sequence $(A_n)_{n \ge 1}$ tending to infinity such that $\log A_n = o(\log n)$. Let us write for every $n \ge 1$, 
\begin{multline*}
\Es{\left(\frac{1}{\log n} \sum_{i=1}^n X^{(\varepsilon)}_i\right)^2}
\\
\le \frac{1}{(\log n)^2} \sum_{i=1}^n \Es{(X^{(\varepsilon)}_i)^2}
+ \frac{2}{(\log n)^2} \sum_{i=1}^n \sum_{j=i+1}^{A_n i} \left|\Es{X^{(\varepsilon)}_i X^{(\varepsilon)}_j}\right|
+ \frac{2}{(\log n)^2} \sum_{i=1}^n \sum_{j=A_n i}^n \left|\Es{X^{(\varepsilon)}_i X^{(\varepsilon)}_j}\right|.
\end{multline*}
The first term converges to zero as $n \to \infty$ since
\[\Es{(X^{(\varepsilon)}_i)^2}
\le \Es{\left(\frac{1}{P_i}\right)^2 \ind{\varepsilon i \le P_i \le i / \varepsilon}}
\le \left(\frac{1}{i \varepsilon}\right)^2.
\]
Appealing to Cauchy--Schwarz inequality, we have for every $i$ large enough,
\[\sum_{j=i+1}^{A_n i} \left|\Es{X^{(\varepsilon)}_i X^{(\varepsilon)}_j}\right|
\le \sum_{j=i+1}^{A_n i} \sqrt{\left(\frac{1}{i \varepsilon}\right)^2 \left(\frac{1}{j \varepsilon}\right)^2}
\le \frac{1}{\varepsilon^2} \frac{2 \log A_n}{i}
= \frac{o(\log n)}{i},\]
and so the second term above converges to zero as $n \to \infty$ as well. Controlling the last term is more involved. In Lemma \ref{lem:decorrelation_FPP}, we shall prove that
\[\lim_{\varepsilon \downarrow 0} \limsup_{i \to \infty} \sup_{1 \le i < n} \sup_{A_n i \le j \le n} 
\left|\Es{(i X^{(\varepsilon)}_i) \cdot (j X^{(\varepsilon)}_j)}\right| = 0.\]
Since we have $\sum_{i=1}^n \sum_{j=A_n i}^n \frac{1}{ij} \le \sum_{i=1}^n \sum_{j=1}^n \frac{1}{ij} = O((\log n)^2)$, it indeed follows that
\[\lim_{\varepsilon \downarrow 0}  \limsup_{n \to \infty} \frac{1}{(\log n)^2} \sum_{i=1}^n \sum_{j=A_n i}^n \left|\Es{X^{(\varepsilon)}_i X^{(\varepsilon)}_j}\right|
= 0,\]
which concludes the proof.
\end{proof}

The next lemma has been used in the course of the previous proof

\begin{lem}\label{lem:decorrelation_FPP}
For every sequence $(A_n)_{n \ge 1}$ tending to infinity, we have
\[\lim_{\varepsilon \downarrow 0} \limsup_{i \to \infty} \sup_{1 \le i < n} \sup_{A_n i \le j \le n} 
\left|\Es{(i X^{(\varepsilon)}_i) \cdot (j X^{(\varepsilon)}_j)}\right| = 0,\]
where $X^{(\varepsilon)}_i$ is as in \eqref{eq:epsilon_cut_inverse_perimeter}.
\end{lem}

We rely on the Markov property of the process $P = (P_k)_{k \ge 0}$ and the estimate from Corollary \ref{cor:bound_conditional_LLT} which shows that if $j$ is far enough from $i$, then $P_j$ does not depend crucially on $P_i$. This shows that, when $j$ is far enough from $i$, the law of $(P_i, P_j)$ is close to that of $(P_i, P_{j-i}')$, which is itself close to that of $(P_i, P_j')$, where $P'$ is an independent copy of $P$, and we conclude that the covariance between $X^{(\varepsilon)}_i$ and $X^{(\varepsilon)}_j$ is small.

\begin{proof}
Throughout the proof, we shall write $C$ for a universal constant, which may change from one line to another. Let us write
\begin{equation}\label{eq:bound_covariance_fpp}
\begin{aligned}
\left|\Es{(i X^{(\varepsilon)}_i) (j X^{(\varepsilon)}_j)}\right| \le &\left|\Es{(i X^{(\varepsilon)}_i) (j X^{(\varepsilon)}_j) \ind{P_i \notin [i \varepsilon, i/\varepsilon]}}\right|
\\
+ &\left|\Es{(i X^{(\varepsilon)}_i) (j X^{(\varepsilon)}_j) \ind{P_j \notin [j \varepsilon, j/\varepsilon]}}\right|
\\
+ &\left|\Es{(i X^{(\varepsilon)}_i) (j X^{(\varepsilon)}_j) \ind{P_i \in [i \varepsilon, i/\varepsilon]} \ind{P_j \in [j \varepsilon, j/\varepsilon]}}\right|,
\end{aligned}
\end{equation}
and let us consider each term separately. We put $K^{(\varepsilon)}_i = \E[P_i^{-1} \ind{P_i \in [\varepsilon i, \varepsilon^{-1} i]}]$, then from Lemma \ref{lem:bound_moments_perimeter_FPP}, we have $\sup_{i \ge 1, \varepsilon \in (0,1)} iK^{(\varepsilon)}_i < C < \infty$ and furthermore the family $(i X^{(\varepsilon)}_i)_{i \ge 1, 0 < \varepsilon < 1}$ is uniformly integrable. Recall also that $i/P_i$ converges in distribution towards $(\mathsf{p}_\q \Upsilon^\uparrow_1)^{-1}$ as $i \to \infty$, so, in particular, $\sup_{i \ge 1} \P(P_i > i/\varepsilon)$ converges to $0$ as $\varepsilon \to 0$. For the first term of \eqref{eq:bound_covariance_fpp}, note that on the event $P_i \notin [i \varepsilon, i/\varepsilon]$, we have $X^{(\varepsilon)}_i = - K^{(\varepsilon)}_i$ so
\[\left|\Es{(i X^{(\varepsilon)}_i) (j X^{(\varepsilon)}_j) \ind{P_i \notin [i \varepsilon, i/\varepsilon]}}\right|
\le (i K^{(\varepsilon)}_i)  \left|\Es{j X^{(\varepsilon)}_j \ind{P_i \notin [i \varepsilon, i/\varepsilon]}}\right|,\]
which converges to $0$ as $\varepsilon \to 0$ uniformly in $i,j \ge 1$ by the previous uniform integrability. The second term in \eqref{eq:bound_covariance_fpp} is symmetric, it only remains to prove
\[\lim_{\varepsilon \downarrow 0} \limsup_{n \to \infty} \sup_{1 \le i < n} \sup_{A_n i \le j \le n} 
\left|\Es{(i X^{(\varepsilon)}_i) (j X^{(\varepsilon)}_j) \ind{P_i \in [i \varepsilon, i/\varepsilon]} \ind{P_j \in [j \varepsilon, j/\varepsilon]}}\right| = 0.\]

Recall that $(A_n)_{n \ge 1}$ is increasing, and assume that $A_1 \ge 2$ for simplicity. Then for every $1 \le i < j \le n$, we have $A_n \ge A_{j-i}$, and if $j \ge A_n i$, then $1/2 \le (j-i)/j \le 1$. We apply the bound from Corollary \ref{cor:bound_conditional_LLT} with $\varepsilon$, $n$, and $\delta_n$ replaced respectively by $\varepsilon/2$, $j-i$, and $2 / (\varepsilon A_{j-i})$, and with the $o(1)$ smaller than $\varepsilon^3$ to obtain: for every $i$ large enough and every $A_n i \le j \le n$,
\begin{align*}
\sup_{\substack{x \le i/\varepsilon\\ \varepsilon j \le y \le j/\varepsilon}}
\left|\P_x(P_{j-i} = y) - \P_1(P_{j-i} = y)\right|
&\le
\sup_{\substack{x \le 2(j-i)/(\varepsilon A_{j-i})\\ \varepsilon (j-i)/2 \le y \le 2(j-i)/\varepsilon}}
\left|\P_x(P_{j-i} = y) - \P_1(P_{j-i} = y)\right|
\\ &\le
j^{-1} \varepsilon^{5/2}.
\end{align*}
The Markov property then yields
\begin{align*}
\left|\Prc{P_j =y}{P_i} - \Prc{P_j = y}{P_i=1}\right| \ind{P_i \le i/\varepsilon}
&= \sum_{x \le i/\varepsilon} \left|\P_x(P_{j-i}  = y) - \P_1(P_{j-i} = y)\right| \ind{P_i =x}
\\
&\le j^{-1} \varepsilon^{5/2} \ind{P_i \le i/\varepsilon},
\end{align*}
uniformly for $\varepsilon j \le y \le j/\varepsilon$ and so
\begin{align*}
&\ind{P_i \le i/\varepsilon} \left|\Esc{(j X^{(\varepsilon)}_j) \ind{P_j \in [j \varepsilon, j/\varepsilon]}}{P_i} - \Esc{(j X^{(\varepsilon)}_j) \ind{P_j \in [j \varepsilon, j/\varepsilon]}}{P_i=1}\right|
\\
&\le \ind{P_i \le i/\varepsilon} \sum_{j \varepsilon \le y \le j/\varepsilon} \left|\frac{j}{y} - jK^{(\varepsilon)}_j\right| \frac{\varepsilon^{5/2}}{j}
\\
&\le \ind{P_i \le i/\varepsilon} \left(\frac{1}{\varepsilon} + jK^{(\varepsilon)}_j\right) \varepsilon^{3/2}.
\end{align*}
Moreover, 
\[\left|\Esc{(j X^{(\varepsilon)}_j) \ind{P_j \in [j \varepsilon, j/\varepsilon]}}{P_i=1}\right|
\le \Es{|j X^{(\varepsilon)}_{j-i}|}
\le C.\]
Using the fact that $X^{(\varepsilon)}_i$ is centred, we also have
\begin{align*}
\left|\Es{(i X^{(\varepsilon)}_i) \ind{P_i \in [i \varepsilon, i/\varepsilon]}}\right|
&= \left|\Es{(i X^{(\varepsilon)}_i) \ind{P_i \notin [i \varepsilon, i/\varepsilon]}}\right|
\\
&\le (i K^{(\varepsilon)}_i) \Pr{P_i \notin [i \varepsilon, i/\varepsilon]}
\\
&\le C \cdot \Pr{\Upsilon^\uparrow_1 \notin [\varepsilon, 1/\varepsilon]}.
\end{align*}
Hence
\[\left|\Es{(i X^{(\varepsilon)}_i) (j X^{(\varepsilon)}_j) \ind{P_i \in [i \varepsilon, i/\varepsilon]} \ind{P_j \in [j \varepsilon, j/\varepsilon]}}\right|
\le C \cdot \Pr{\Upsilon^\uparrow_1 \notin [\varepsilon, 1/\varepsilon]} \left(C + \left(\frac{1}{\varepsilon} + C\right) \varepsilon^{3/2}\right),\]
which converges to $0$ as $\varepsilon \to 0$.
\end{proof}

\section{Graph distance}
\label{sec:dual}

The goal of this section is to prove our main result Theorem \ref{thm:intro_dual}. Recall that we denote by $\hBall(\map^{\dagger})$ the hull of the ball of radius $r \ge 1$ centered at the root-vertex for the graph distance in $\map^{\dagger}$. As in the previous section, the proof of this result is based on the peeling process of $ \map$ with an algorithm --now deterministic--  especially designed to discover (hulls of) metric balls in the dual map $\map^{\dagger}$, which we now recall. The latter is taken from \cite[Section 2.3]{Budd-Curien:Geometry_of_infinite_planar_maps_with_high_degrees}, see also \cite{Curien-Le_Gall:Scaling_limits_for_the_peeling_process_on_random_maps} for a slightly different type of peeling exploration (the ``face-peeling''). \medskip

Let $\map$ be an infinite, one-ended, rooted bipartite map. We construct a peeling exploration $(\overline{\mathfrak{e}}_i)_{i \ge 0}$ as follows. First, recall that $\overline{\mathfrak{e}}_0$ consists of a unique simple face with same degree as the root-face of $\map$, and a hole with the same degree. For any edge on the boundary of the hole of $ \overline{ \mathfrak{e}}_{i}$ we call its \emph{height} the graph distance in $ \map^{\dagger}$ between the adjacent face in $ \overline{ \mathfrak{e}}_{i}$ to that edge and the root-face of $\map$ (the root-vertex of $ \map^{\dagger}$). Inductively, suppose that at every step $i \ge 0$, the following hypothesis is satisfied:
\begin{center}
\begin{minipage}{.9\linewidth} $(H)$: There exists an integer $h \geq 0$ such that in the explored map $\overline{\mathfrak{e}}_i$ all the edges on the boundary of the hole are at height $h$ or $h+1$, and that the set of edges at height $h$ forms a connected segment.
\end{minipage}
\end{center} 
If $\overline{\mathfrak{e}}_i$ satisfies $(H)$ then the next edge to peel $\mathcal{A}(\overline{\mathfrak{e}}_i)$ is chosen as follows:
\begin{enumerate}
\item If all edges on the boundary of the hole of $\overline{\mathfrak{e}}_i$ are at the same height $h$ then $\mathcal{A}(\overline{\mathfrak{e}}_i)$ is any (deterministically chosen) edge on the boundary;
\item Otherwise $\mathcal{A}(\overline{\mathfrak{e}}_i)$ is the unique edge at height $h$ such that the edge immediately on its left is at height $h+1$.
\end{enumerate}

\begin{figure}[!ht]
\begin{center}
\includegraphics[width=.6\linewidth]{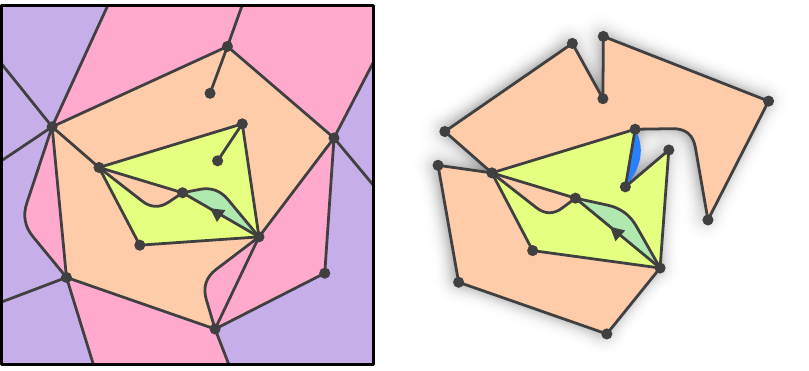}
\caption{ The left figure shows a portion of an infinite planar map with faces coloured according to (dual) graph distance to the root-face. The sub-map on the right depicts a possible state of the peeling by layers. Its hole has three edges at height $h=1$ and eleven edges at height $h+1=2$. The next peel edge is indicated in blue. }
\end{center}
\end{figure}

The first step $\overline{\mathfrak{e}}_0$ satisfies $(H)$ and one easily checks inductively that if we use this algorithm, then for any $i \ge 0$, the $i$-th step $\overline{\mathfrak{e}}_i$ satisfies it as well. The sequence $(\overline{\mathfrak{e}}_i)_{i \ge 0}$ thus forms a well-defined peeling exploration which we call the \emph{peeling by layers} of $ \map$. The name is justified by the following observation: for every $i \ge 0$, let $H_i$ be the minimal height of an edge on the boundary of $\overline{\mathfrak{e}}_i$, so $H_0 = 0$ and for each $i \ge 0$, we have $\Delta H_i \coloneqq H_{i+1}-H_i \in \{0, 1\}$. For each $r \ge 0$, if we denote by $\theta_r = \inf\{i \ge 0 : H_i \ge r\}$ then 
\begin{equation}\label{eq:interpretation_graph_distance}
\overline{\mathfrak{e}}_{\theta_{r}}  \mbox{ is the dual inside $\map$ of } \hBall(\map^{\dagger})
\qquad\text{for every}\qquad
r \ge 0.
\end{equation}
If we denote by $|\partial^{*}  \hBall(\map^\dagger)| \coloneqq |\partial \overline{ \mathfrak{e}}_{\theta}|$ the number of edges adjacent to the hole of $ \hBall(\map^{\dagger})$ then we have the following generalisation of Theorem \ref{thm:intro_dual}:
\begin{thm}[Graph distance growth]\label{thm:graph_distance_growth}
The following convergences in probability hold:
\[\frac{\log |\partial^{*} \hBall(\Map_\infty^\dagger)|}{ \sqrt{r}} \cvproba[r] \pi \sqrt{2}
\qquad\text{and}\qquad
\frac{\log \|\hBall(\Map_\infty^\dagger)\|}{ \sqrt{r}} \cvproba[r] \frac{3 \pi}{\sqrt{2}}.\]
\end{thm}

Theorem \ref{thm:graph_distance_growth} then follows from Proposition \ref{prop:limit_height_dual} below  together with Theorem \ref{thm:convergence_processes_perimeter_volume} in the same way that Theorem \ref{thm:perimeter_volume_fpp} followed from Proposition \ref{prop:convergence_jump_times_FPP} together with Theorem \ref{thm:convergence_processes_perimeter_volume}.

\begin{prop}\label{prop:limit_height_dual}
The following convergence in probability holds:
\[\frac{H_n}{(\log n)^2} \cvproba \frac{1}{2\pi^2}.\]
\end{prop}

\begin{proof}[Proof of Theorem \ref{thm:graph_distance_growth}]
Let $(P_i, V_i)_{i \ge 0}$ be the half-perimeter and volume process associated with this peeling by layers. Then we have for all $r \ge 1$
\[\left(\frac{\log |\partial \hBall(\Map_\infty)|}{\sqrt{r}}, \frac{\log |\hBall(\Map_\infty)|}{\sqrt{r}}\right)
= \frac{\log \theta_r}{\sqrt{r}} \cdot \left(\frac{\log P_{\theta_r}}{\log \theta_r}, \frac{\log V_{\theta_r}}{\log \theta_r}\right).\]
Since $H_n$ is non-decreasing, Proposition \ref{prop:limit_height_dual} implies that $r^{-1/2} \log \theta_r \to \pi \sqrt{2}$ in probability as $r \to \infty$ and we conclude appealing to Lemma \ref{lem:as_convergence_perimeter_volume}.
\end{proof}

The rest of the section is devoted to proving the key Proposition \ref{prop:limit_height_dual}. Before diving into the proof, let us present the heuristics. We can imagine the peeling by layers algorithm as ``turning around'' $\overline{\mathfrak{e}}_{n}$ in clockwise order to discover, layer after layer the faces of $\map$ by increasing dual graph distance (and filling-in the holes created). The idea is to consider the speed at which the algorithm is turning around the current explored map. Compared to the variation of the perimeter process $P_{n}$ which is a balanced random walk, the speed at which the peeling by layer turns around $\overline{\mathfrak{e}}_{n}$ is given by an \emph{unbalanced} random walk with Cauchy-type tails. By standard results, the rate of growth of such random walks is of order $n \log n$. Since $P_{n} \approx n$, the last heuristic suggests that the time needed to complete a layer around $ \mathfrak{e}_{n}$ is of order $ n/ \log n$ and so typically we should have 
$$ \Delta H_{n} \approx \frac{\log n}{ n}.$$ Integrating the last display indeed shows that $H_{n} \approx \log ^{2} n$.

More precisely, the proof of Proposition \ref{prop:limit_height_dual} is divided into two parts, which use different arguments. We first prove in the next subsection an upper-bound showing that $\limsup_{n \to \infty} (\log n)^{-2} \Es{H_n} \le 1/(2\pi^2)$. This part is based on a computational trick to estimate $ \mathbb{E}[\Delta H_{n}]$.  Section \ref{sec:lower_bound_height} is devoted to proving a lower bound showing  that for every $\varepsilon > 0$, we have $\lim_{n \to \infty} \P((\log n)^{-2} H_n \ge (1-\varepsilon)/(2\pi^2)) = 1$. This part uses ideas presented in above the heuristic (see in particular Lemma \ref{lem:time_complete_first_layer}).  It is then straightforward to combine the upper and lower bounds to prove Proposition \ref{prop:limit_height_dual}. We leave the details to the reader.

\subsection{The upper bound}
\label{sec:upper_bound_height}

\begin{prop}\label{prop:upper_bound_expectation_H}
We have
\[\limsup_{n \to \infty} \frac{1}{(\log n)^2} \Es{H_n}
\le \frac{1}{2\pi^2}.\]
\end{prop}

The proof relies on a technical estimate of an interpolated version of the sequence $(H_n)_{n \ge 0}$. Let $D_n$ be the number of edges in the boundary that are at height $H_n$. The other $2P_n - D_n$ edges are at height $H_{n}+1$.
For a non-increasing function $f:[0,1]\to[0,1]$, we introduce the random sequence
\[H_n^f \coloneqq H_n + f\left( \frac{D_n}{2P_n}\right) \quad \mbox{and} \quad \Delta H_n^f \coloneqq H_{n+1}^f - H_n^f.\]
The next result is inspired by \cite[Lemma 13]{Curien-Le_Gall:Scaling_limits_for_the_peeling_process_on_random_maps} in the case of random triangulations and by \cite[Lemma 4.3]{Budd-Curien:Geometry_of_infinite_planar_maps_with_high_degrees} in the case of Boltzmann maps with exponent $a \in (2, 5/2)$. There the function $f: x\mapsto 1-x$ was used to obtain an upper-bound on $\Es{H_n} \leq \Es{H_n^f}$ with an arbitrary multiplicative constant.
In order to obtain the correct constant one has to choose the function $f$ more carefully (see Figure \ref{fig:interpolation_H}).

\begin{lem}\label{lem:bound_expectation_H_interpolated}
If $f: [0,1] \to [0,1]$ is twice continuously differentiable with $f(0) = 1$, $f(1)=0$, $f''(0)=f''(1)=0$ and $0 \leq -f'(x) \leq 1+\varepsilon$ for some $\varepsilon\in(0,1)$ and all $x\in[0,1]$, then there exists a $C>0$ such that for all $n \geq 1$ we have
	\[ \Esc{\Delta H_n^f}{P_n} \leq (1+3\varepsilon)^3 \frac{\mathsf{p}_\q}{2} \frac{\log(P_n)+C}{P_n}. \]
\end{lem}

\begin{figure}[!ht]
\begin{center}
\includegraphics[width=.5\linewidth]{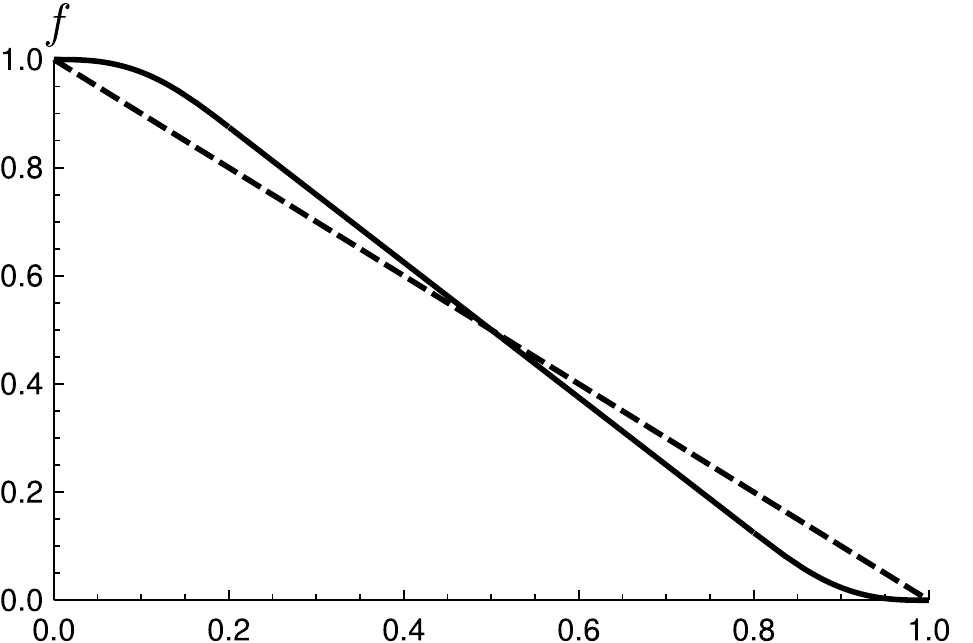}
\caption{The solid curve corresponds to a function $f$ satisfying the hypotheses of Lemma \ref{lem:bound_expectation_H_interpolated} with $\varepsilon=0.2$, while the dashed lined depicts $x\mapsto1-x$.}
\label{fig:interpolation_H}
\end{center}
\end{figure}

\begin{proof}
	We will establish the equivalent bound on the conditional expectation $\E[\Delta H_n^f \mid P_n=p, D_n=\ell]$ uniformly in $1\leq \ell\leq 2p$.
	Unless stated otherwise the correction terms $O(p^{-1})$ will be uniform in $\ell$.
	For convenience we will use the notation $y \coloneqq \ell/(2p) \in (0,1]$ and we will extend $f$ to a function $\R\to[0,1]$ by setting $f(x) = 1$ for $x<0$ and $f(x)=0$ for $x>1$. 
	
	We claim that $\E[\Delta H_n^f \mid P_n=p, D_n=\ell]$ can be expressed as
	\begin{align} \Esc{\Delta H_n^f}{P_n=p, D_n=\ell } &= \sum_{k=0}^\infty\nu(k) \frac{h^\uparrow(p+k)}{h^\uparrow(p)} \left[ f\left(\frac{\ell-1}{2p+2k}\right)-f\left(\frac{\ell}{2p}\right)\right] \label{eq:deltaH1}\\
	&+\sum_{k =1}^{p-1}\frac{\nu(-k)}{2} \frac{h^\uparrow(p-k)}{h^\uparrow(p)} \left[ f\left(\frac{\ell-2k}{2p-2k}\right)-f\left(\frac{\ell}{2p}\right)\right]\label{eq:deltaH2}\\
	&+\sum_{k =1}^{p-1}\frac{\nu(-k)}{2} \frac{h^\uparrow(p-k)}{h^\uparrow(p)} \left[  f\left(\frac{\ell-1}{2p-2k}\right)-f\left(\frac{\ell}{2p}\right)\right].\label{eq:deltaH3}
	\end{align}
	Using the law of the Markov process $(P_n,D_n,H_n)$, which was explicitly identified in \cite[Section 4.2]{Budd-Curien:Geometry_of_infinite_planar_maps_with_high_degrees}, this can be checked as follows.
	First of all, if $\ell=1$ then necessarily $\Delta H_{n+1} = 1$ and $D_{n+1} = 2P_{n+1}$. Therefore $\Delta H_n^f = 1-f(1/(2p))$, which agrees with each of the terms in brackets since we set $f(x)=1$ for $x\leq 0$, and hence with the sum. 
	Suppose now $\ell \geq 2$.
	If a new face of half-degree $k+1$ is discovered, then $P_{n+1} = p+k$ and $D_{n+1} = \ell-1$, which explains the first line (\ref{eq:deltaH1}).	
	If the peel edge is identified with an edge on its right, then $P_{n+1} = p-k$ for some $1\leq k<p$ and either $D_{n+1} = \ell - 2k$ and $\Delta H_{n} = 0$ or $D_{n+1}=2P_{n+1}$ and $\Delta H_{n} = 1$, depending on the sign of $\ell - 2k$. 
	The second line (\ref{eq:deltaH2}) incorporates both these situations.
	Finally, if the peel edge is identified with an edge on its left, then $P_{n+1} = p-k$ for some $1\leq k<p$ and $\Delta H_n=0$ and either $D_{n+1} = \ell - 1$ or $D_{n+1}=2P_{n+1}$, depending on which is smallest.
	Both situations are taken into account in the last line (\ref{eq:deltaH3}).
	
	To bound the expectation value we rely on the following estimates, which are easily checked:
	\begin{itemize}
		\item By Taylor's theorem there is a $c>0$ such that $|f(y)-f(x)-f'(x)(y-x)| < c (y-x)^2$ for all $x,y\in\R$.
		\item There exists an $N>0$ such that $(1-\varepsilon) < |k|^2\nu(k)/\mathsf{p}_\q < (1+\varepsilon) $ for all $|k|> N$.
		\item For any $k > -p$ we have 
		\[ \sqrt{1+\frac{k}{p}}\left(1-\frac{1}{p}\right) \leq \frac{h^\uparrow(p+k)}{h^\uparrow(p)} \leq \sqrt{1+\frac{k}{p}}\left(1+\frac{1}{p}\right).\]
	\end{itemize}
	The first $N$ terms in each of the sums together contribute $O(p^{-1})$ (uniformly in $\ell$), so it suffices to bound the sums restricted to $|k| \geq N$.
	For simplicity we may also assume $p > 2N$.
	
	Starting with the first sum (\ref{eq:deltaH1}), which has only positive terms, the aforementioned estimates lead to
	\begin{align*}
	\sum_{k=N}^\infty\nu(k) \frac{h^\uparrow(p+k)}{h^\uparrow(p)} &\left[ f\left(\frac{\ell-1}{2p+2k}\right)-f\left(\frac{\ell}{2p}\right)\right] \\
	&\leq (1+\varepsilon)\left(1+\frac{1}{p}\right) \sum_{k=N}^\infty \frac{\mathsf{p}_\q}{k^2}\sqrt{1+\frac{k}{p}}\left[|f'(y)|\frac{1+yk}{p+k}+c \left(\frac{1+k}{p+k}\right)^2 \right],
	\end{align*}
	where we used that $\frac{\ell}{2p}-\frac{\ell-1}{2p+2k} \leq \frac{1+yk}{p+k} \leq \frac{1+k}{p+k}$.
	The sum of the quadratic term is $O(p^{-1})$ and therefore the full sum is bounded by
	\[ (1+\varepsilon)\mathsf{p}_\q\, y |f'(y)|  \, \frac{1}{p} \int_{\frac{N-1}{p}}^\infty \frac{\mathrm{d}x}{x\sqrt{1+x}} + O(p^{-1}) \leq (1+\varepsilon)\mathsf{p}_\q\, y|f'(y)|\, \frac{\log(p)}{p} + O(p^{-1}). \]
	
	The second sum (\ref{eq:deltaH2}) also has positive terms, so
	\begin{align*}
	\sum_{k=N}^{p-1}\frac{\nu(-k)}{2} \frac{h^\uparrow(p-k)}{h^\uparrow(p)} & \left[ f\left(\frac{\ell-2k}{2p-2k}\right)-f\left(\frac{\ell}{2p}\right)\right] \\
	&\leq \frac{(1+\varepsilon)}{2}\left(1+\frac{1}{p}\right) \sum_{k=N}^{p-1} \frac{\mathsf{p}_\q}{k^2}\sqrt{1-\frac{k}{p}}\left[|f'(y)|(1-y)\frac{k}{p}+c \left(\frac{k}{p}\right)^2 \right],
	\end{align*}
	where we used that  $\frac{\ell}{2p}-\frac{\ell-2k}{2p-2k} \leq \frac{(1-y)k}{p} \leq \frac{k}{p}$.
	Again the quadratic term yields a contribution $O(p^{-1})$ and therefore the full sum is bounded by
	\[ (1+\varepsilon)\mathsf{p}_\q \, \frac{1-y}{2}|f'(y)| \,\frac{1}{p}  \int_{\frac{N-1}{p}}^1 \frac{\sqrt{1-x}}{x}\mathrm{d}x + O(p^{-1}) \leq (1+\varepsilon)\mathsf{p}_\q \, \frac{1-y}{2}|f'(y)| \frac{\log(p)}{p} + O(p^{-1}). \]
	
	The last sum (\ref{eq:deltaH3}) may have both positive and negative terms (depending on the sign of $p-\ell k$), so we should proceed more carefully.
	Since the terms with $p/2 \leq k \leq p-1$ are $O(p^{-2})$ uniformly in $k$ and $\ell$, it suffices to consider the sum restricted to $N \leq k \leq p/2$.
	For these values of $k$, $\frac{\ell}{2p}-\frac{\ell-1}{2p-2k} \leq \frac{1}{p}$ and therefore the contribution of the positive terms to the sum is $O(p^{-1})$. 
	
	We now distinguish two cases: $y \in (0,\varepsilon)$ and $y \in [\varepsilon,1]$.
	In the first case we are done, because taking into account just the positive contributions already yields the uniform bound	
	\begin{align*}
	\Es{\Delta H_n^f \,\big| P_n=p, D_n=\ell } &\leq \frac{\mathsf{p}_\q}{2} (1+\varepsilon)(1+y)|f'(y)|\, \frac{\log(p)}{p} + O(p^{-1}) \leq (1+\varepsilon)^3 \frac{\mathsf{p}_\q}{2} \frac{\log(p)}{p} + O(p^{-1}).
	\end{align*}
	
	In the case $y \in [\varepsilon,1]$ we still need a lower-bound on the absolute value of the negative terms. 
	Assuming we have chosen $N > 1/\varepsilon$ all terms with $k \geq N \geq 2p/\ell$ are negative and we have
	\begin{align*}
	\sum_{k=N}^{p/2}\frac{\nu(-k)}{2} \frac{h^\uparrow(p-k)}{h^\uparrow(p)} &\left| f\left(\frac{\ell-1}{2p-2k}\right)-f\left(\frac{\ell}{2p}\right)\right| \\
	&\geq \frac{(1-\varepsilon)}{2}\left(1-\frac{1}{p}\right) \sum_{k=N}^{p/2} \frac{\mathsf{p}_\q}{k^2} \sqrt{1-\frac{k}{p}}\left[|f'(y)|\left(y\frac{k}{p}-\frac{1}{2p}\right)-c \left(\frac{2k}{p}\right)^2 \right],
	\end{align*}
	where we used that $\frac{\ell-1}{2p-2k}-\frac{\ell}{2p} \geq \frac{\ell k-p}{2p^2} = y\frac{k}{p} - \frac{1}{2p}$ and $|\frac{\ell}{2p}-\frac{\ell-1}{2p-2k}| \leq 2k/p$.
	As before the quadratic part yields a contribution $O(p^{-1})$ and the full sum is bounded from below by
	\[
	(1-\varepsilon) \mathsf{p}_\q\, \frac{y}{2}|f'(y)|\,\frac{1}{p}  \int_{N/p}^{1/2} \frac{\sqrt{1-x}}{x}\mathrm{d}x + O(p^{-1}) \geq (1-\varepsilon) \mathsf{p}_\q\, \frac{y}{2}|f'(y)|\, \frac{\log(p)}{p} + O(p^{-1}).
	\]
	Putting the three sums together and using that $y\in [\varepsilon,1]$ and $|f'(y)| \leq 1+\varepsilon$ we find
	\begin{align*}
	\Esc{\Delta H_n^f}{P_n=p, D_n=\ell } &\leq \frac{\mathsf{p}_\q}{2} ( (1+\varepsilon)(1+y) - (1-\varepsilon)y )|f'(y)|\, \frac{\log(p)}{p} + O(p^{-1}) \\
	&\leq (1+3\varepsilon)^2 \frac{\mathsf{p}_\q}{2} \frac{\log(p)}{p} + O(p^{-1}),
	\end{align*}
	uniformly in $1\leq \ell\leq 2p$.
\end{proof}

The proof of Proposition \ref{prop:upper_bound_expectation_H} is now a simple consequence of Lemma \ref{lem:bound_expectation_H_interpolated}.

\begin{proof}[Proof of Proposition \ref{prop:upper_bound_expectation_H}]
It is not hard to see that for any $\varepsilon \in (0,1)$ there exists a function $f$ satisfying the conditions of Lemma \ref{lem:bound_expectation_H_interpolated}. 
Since $H_n \leq H_n^f$ we immediately obtain that for any $\varepsilon \in (0,1)$ there exists a $C>0$ such that
\begin{equation}\label{eq:Hnestimate} 
\Es{H_n} \leq \Es{H_n^f} = \sum_{k=0}^{n-1} \Es{\Delta H_k^f} \leq (1+2\varepsilon)^2 \frac{\mathsf{p}_\q}{2} \sum_{k=0}^{n-1} \Es{\frac{\log(P_k)+C}{P_k}}.
\end{equation}
We claim that for any $C > 0$,
\[\lim_{k \to \infty} \frac{k}{\log k} \Es{\frac{\log(P_k)+C}{P_k}}
= \Es{(\mathsf{p}_\q \Upsilon^\uparrow_1)^{-1}}
= \frac{2}{\pi^2 \mathsf{p}_\q}.\]
Recall from Theorem \ref{thm:convergence_processes_perimeter_volume} that $k/P_k$ converges in law to $(\mathsf{p}_\q \Upsilon^\uparrow_1)^{-1}$ as $k \to \infty$ so $(k (\log(P_k)+C))/(P_k \log k)$ as well. Let us prove that this sequence is uniformly integrable. First, according to Lemma \ref{lem:bound_moments_perimeter_FPP}, the sequence $(k/P_k)_{k \ge 0}$ is bounded in $\mathrm{L}^q$ for all $q < 3/2$, and so $((C k)/(P_k \log k))_{k \ge 2}$ as well. Next, fix $\varepsilon \in (0,1)$, and observe that since $x \mapsto x/\log x$ is increasing on $[\ex, \infty)$, we have
\[\frac{k \log P_k}{P_k \log k} \ind{\frac{k \log P_k}{P_k \log k} \ge 1/\varepsilon}
\le \frac{k \log P_k}{P_k \log k} \ind{\frac{P_k}{\log P_k} \le \frac{\varepsilon k}{\log(\varepsilon k)}}
\le \frac{k \log P_k}{P_k \log k} \ind{P_k \le \varepsilon k}
\le \frac{k}{P_k},\]
and the right-hand side is bounded in $\mathrm{L}^q$ for all $q < 3/2$.

The bound \eqref{eq:Hnestimate} and the fact that $(\log n)^{-2} \sum_{k=1}^n \log(k)/k \to 1/2$ as $n \to \infty$ then yield
\[\limsup_{n \to \infty} \frac{1}{(\log n)^2} \Es{H_n}
\le \frac{\mathsf{p}_\q}{2} \cdot \frac{2}{\pi^2 \mathsf{p}_\q} \cdot \frac{1}{2},\]
which concludes the proof.
\end{proof}

\subsection{The lower bound}
\label{sec:lower_bound_height}

In this subsection, we aim at showing the following result, which completes the proof of Proposition \ref{prop:limit_height_dual}.

\begin{prop}\label{prop:lower_bound_height_dual}
For every $\varepsilon > 0$, we have
\[\Pr{\frac{H_n}{(\log n)^2} \ge \frac{1-\varepsilon}{2 \pi^2}} \cv 1.\]
\end{prop}

Recall that for all $k \ge 1$, we denote by $\theta_k$ the first time $n$ when all the edges on the boundary after $n$ peeling steps are at height at least $k$, for the peeling algorithm introduced earlier in this section. Recall also from Section \ref{sec:Doob_transform} the law $\P_\infty^{(\ell)}$ of an infinite rooted bipartite Boltzmann map with a root-face of degree $2\ell$, for $\ell \ge 1$. The main ingredient of the proof of Proposition \ref{prop:lower_bound_height_dual} is the following.

\begin{lem}\label{lem:time_complete_first_layer}
For every $\varepsilon \in (0,1)$,
\[\lim_{\ell \to \infty} \P_\infty^{(\ell)}\left(\frac{\mathsf{p}_\q \log \ell}{2\ell} \cdot \theta_1 \le 1+\varepsilon \quad\text{and}\quad \frac{P_{\theta_1}}{\ell} \in [1-\varepsilon, 1+\varepsilon]\right) = 1.\]
\end{lem}

\begin{rem}\label{rem:time_complete_first_layer}
This lemma will suffice to prove Proposition \ref{prop:lower_bound_height_dual} and so Proposition \ref{prop:limit_height_dual}. However it is not difficult to see (e.g. with the argument developed below) that such an upper bound on $\theta_1$ and the limit of the height $H_n$ from Proposition \ref{prop:limit_height_dual} imply that for every $\varepsilon \in (0,1)$,
\[\lim_{\ell \to \infty} \P_\infty^{(\ell)}\left(\frac{\mathsf{p}_\q \log \ell}{2\ell} \cdot \theta_1 \in[1-\varepsilon, 1+\varepsilon] \quad\text{and}\quad \frac{P_{\theta_1}}{\ell} \in [1-\varepsilon, 1+\varepsilon]\right) = 1.\]
In words, the above display indeed confirms the heuristic presented at the beginning of this section, stating that the time needed for the peeling by layers to complete a full turn around a boundary of perimeter $2 \ell$ is of order $  \frac{2 \ell }{ \mathsf{p}_{ \mathbf{q}} \log \ell} $.
\end{rem}

The last lemma has an immediate corollary on the time needed to complete $k$ layers:

\begin{cor}\label{cor:time_complete_k_layers}
For every $\varepsilon \in (0,1)$ and every $k \in \N$,
\[\lim_{\ell \to \infty} \P_\infty^{(\ell)}\left(\frac{\mathsf{p}_\q \log \ell}{2\ell} \cdot \frac{\theta_k}{k} \le 1+\varepsilon\right) =1.\]
\end{cor}

\begin{proof}
Let us set $\varepsilon \theta_i = \theta_{i+1}-\theta_i$ for every $i \ge 0$. By the strong Markov property of the peeling exploration applied successively at the stopping times $\theta_{k-1}, \theta_{k-2}, \dots, \theta_1$, we deduce from the last lemma that for any $\varepsilon>0$ we have 
\[\lim_{\ell \to \infty} \P_\infty^{(\ell)}\left(\bigcap_{i=0}^{k-1} \left\{\left\{\varepsilon \theta_i \le \frac{(1+\varepsilon) 2P_{\theta_{i}}}{\mathsf{p}_\q \log P_{ \theta_{i}}}\right\} \cap \left\{\frac{P_{\theta_{i+1}}}{P_{\theta_i}} \in[1- \varepsilon, 1+\varepsilon]\right\}\right\}\right) =1.\]
On the event considered in the last display we have $\theta_{k} \leq (1 + C_{\varepsilon}) k \frac{2\ell}{ \mathsf{p}_{ \mathbf{q}} \log \ell }$ for some constant $C_{ \varepsilon}>0$ tending to $0$ as $ \varepsilon \to 0$. This proves the corollary.
\end{proof}

We defer the proof of Lemma \ref{lem:time_complete_first_layer} and first complete the proof of Proposition \ref{prop:lower_bound_height_dual}. The idea goes as follows: fix some $C > 0$ large and set for every $k \ge 1$ a time $t_k \coloneqq \lfloor \exp((2 C k)^{1/2})\rfloor$ so 
\begin{equation}\label{eq:t_k}
t_{\lfloor(\log n)^2/(2C)\rfloor} = n (1+o(1)),
\qquad\text{and}\qquad
\Delta t_k \coloneqq t_{k+1} - t_k = (C+o(1)) \frac{t_k}{\log t_k}.
\end{equation}
Recall from Remark \ref{rem:moment_inverse_Cauchy_positive} that $\E[t/P_t] \to \E[1/(\mathsf{p}_\q \Upsilon^\uparrow_1)] = 2/(\pi^2 \mathsf{p}_\q)$ as $t \to \infty$. Observe next that 
\[\Delta t_k = C\frac{\mathsf{p}_\q t_k}{2 P_{t_k}} \frac{2 P_{t_k}}{\mathsf{p}_\q \log P_{t_k}} \cdot \frac{\log P_{t_k}}{\log t_k} \cdot (1+o(1)).\]
The last factor $\log P_{t_k} / \log t_k$ converges to $1$ in probability from Theorem \ref{thm:convergence_processes_perimeter_volume} (even almost surely from Lemma \ref{lem:as_convergence_perimeter_volume}), whilst, according to Corollary \ref{cor:time_complete_k_layers}, the first factor is roughly bounded below by the time needed for our peeling algorithm to reveal the first $C \mathsf{p}_\q t_k / 2 P_{t_k}$ layers starting from an initial half-perimeter $P_{t_k}$. Therefore, in $\Delta t_k$ amount of time, we discover roughly $H_{t_{k+1}} - H_{t_k} \gtrsim C \mathsf{p}_\q t_k / 2 P_{t_k}$ layers, and so
\begin{equation}\label{eq:behaviour_H_n}
H_n \gtrsim \sum_{k=1}^{(\log n)^2/(2C)} C \frac{\mathsf{p}_\q t_k}{2 P_{t_k}}
\approx \frac{(\log n)^2}{2C} \cdot C \frac{\mathsf{p}_\q}{2} \cdot \Es{(\mathsf{p}_\q \Upsilon^\uparrow_1)^{-1}}
= \frac{(\log n)^2}{2 \pi^2}.
\end{equation}

The next lemma gives a formal statement of the ``$\approx$'' in \eqref{eq:behaviour_H_n}. The argument is similar to the one used in the proof of Proposition \ref{prop:convergence_jump_times_FPP}.

\begin{lem}\label{lem:approximation_Delta_H_t_k}
Fix any $C > 0$ and define $(t_k)_{k \ge 1}$ as in \eqref{eq:t_k}, then
\[\lim_{\delta \downarrow 0} \limsup_{N \to \infty} \Es{\left|\frac{1}{N} \sum_{k=1}^N \frac{t_k}{P_{t_k}} \ind{P_{t_k}/t_k \in[\delta, \delta^{-1}]} - \frac{2}{\pi^2 \mathsf{p}_\q}\right|}
= 0.\]
\end{lem}

\begin{proof}
First observe that, by monotone convergence,
\[\lim_{\delta \downarrow 0} \Es{(\mathsf{p}_\q \Upsilon^\uparrow_1)^{-1} \ind{\mathsf{p}_\q \Upsilon^\uparrow_1 \in[\delta, \delta^{-1}]}}
= \Es{(\mathsf{p}_\q \Upsilon^\uparrow_1)^{-1}}
= \frac{2}{\pi^2 \mathsf{p}_\q},\]
so it is equivalent to prove
\[\lim_{\delta \downarrow 0} \limsup_{N \to \infty} \Es{\left|\frac{1}{N} \sum_{k=1}^N \frac{t_k}{P_{t_k}} \ind{P_{t_k}/t_k \in[\delta, \delta^{-1}]} - \Es{(\mathsf{p}_\q \Upsilon^\uparrow_1)^{-1} \ind{\mathsf{p}_\q \Upsilon^\uparrow_1 \in[\delta, \delta^{-1}]}}\right|}
= 0.\]
Second, from Theorem \ref{thm:convergence_processes_perimeter_volume} and the boundedness of the random variables, for every $\delta > 0$, we have
\[\lim_{k \to \infty} \Es{\frac{t_k}{P_{t_k}} \ind{P_{t_k}/t_k \in[\delta, \delta^{-1}]}} 
= \Es{(\mathsf{p}_\q \Upsilon^\uparrow_1)^{-1} \ind{\mathsf{p}_\q \Upsilon^\uparrow_1 \in[\delta, \delta^{-1}]}}.\]
Fix $\delta > 0$ and recall the notation from \eqref{eq:epsilon_cut_inverse_perimeter}: $X^{(\delta)}_i = \frac{1}{P_i} \ind{P_i/i \in [\delta, \delta^{-1}]} - \E[\frac{1}{P_i} \ind{P_i/i \in [\delta, \delta^{-1}]}]$; it is therefore sufficient to prove
\[\lim_{\delta \downarrow 0} \limsup_{N \to \infty} \Es{\left(\frac{1}{N} \sum_{k=1}^N t_k X^{(\delta)}_{t_k}\right)^2}
= 0.\]
Since $\E[|i X^{(\delta)}_i|^2] \le \delta^{-2}$, we have
\[\lim_{N \to \infty} \frac{1}{N^2} \sum_{k=1}^N \Es{\left(t_k X^{(\delta)}_{t_k}\right)^2}
= 0.\]
Next, let us fix a sequence $(A_N)_{N \ge 1}$ tending to infinity such that $\log A_N = o(N^{1/2})$ as $N \to \infty$. Observe that $t_\ell \le A_N t_k$ is equivalent to $\ell \le k + (\frac{2 k}{C})^{1/2} \log A_N + \frac{(\log A_N)^2}{2C}$. Cauchy--Schwarz inequality yields $|\E[(i X^{(\delta)}_i) (j X^{(\delta)}_j)]| \le \delta^{-2}$ and so
\begin{align*}
&\limsup_{N \to \infty} \frac{1}{N^2} \sum_{1 \le k < \ell \le N} \ind{t_\ell \le A_N t_k} \left|\Es{(t_k X^{(\delta)}_{t_k}) (t_\ell X^{(\delta)}_{t_\ell})}\right|
\\
&\le \limsup_{N \to \infty} \frac{1}{\delta^2} \frac{1}{N^2} \sum_{k = 1}^N \left(\sqrt{\frac{2 k}{C}} \log A_N + \frac{(\log A_N)^2}{2C}\right)
\\
&=0.
\end{align*}
Finally, from Lemma \ref{lem:decorrelation_FPP},
\[\lim_{\delta \downarrow 0} \limsup_{N \to \infty}
\frac{1}{N^2} \sum_{1 \le k < \ell \le N} \ind{t_\ell > A_N t_k} \left|\Es{(t_k X^{(\delta)}_{t_k}) (t_\ell X^{(\delta)}_{t_\ell})}\right|
= 0.\]
This concludes the proof.
\end{proof}

It remains to give a formal statement of the ``$\gtrsim$'' in \eqref{eq:behaviour_H_n}.

\begin{proof}[Proof of Proposition \ref{prop:lower_bound_height_dual}]
Let us fix $C > 0$ and $\delta \in (0, 1)$ and for every $k \ge 1$, set
\[\eta_k \coloneqq \left(H_{t_{k+1}} - H_{t_k} - C \frac{\mathsf{p}_\q}{2} \frac{t_k}{P_{t_k}}\right) \ind{P_{t_k}/t_k \in[\delta, \delta^{-1}]},\]
where $(t_k)_{k \ge 1}$ is as in \eqref{eq:t_k} and depends on $C > 0$. In words, on the event $P_{t_{k}} \in [\delta t_{k}, \delta^{-1} t_{k}]$, the quantity $\eta_{k}$ is the error we make by approximating $\Delta H_{t_{k}}$ by its value predicted by Corollary \ref{cor:time_complete_k_layers} given $P_{t_{k}}$. We first claim that, $C, \delta > 0$ being fixed, we have $\P(\eta_k \ge -1) \to 1$ as $k \to \infty$. Indeed, let us choose $0 < \varepsilon < 2\delta/(C\mathsf{p}_\q - 2\delta)$ and observe that the following inclusion of events holds:
\begin{equation}\label{eq:simplify_bound_eta}
\left\{\frac{P_{t_k}}{t_k} \notin[\delta, \delta^{-1}]\right\}
\cup
\left\{\left\{\frac{P_{t_k}}{t_k} \in[\delta, \delta^{-1}]\right\} \cap \left\{H_{t_{k+1}} - H_{t_k} \ge \frac{C \mathsf{p}_\q t_k}{(1+\varepsilon) 2P_{t_k}}\right\}\right\}
\subset \{\eta_k \ge -1\}.
\end{equation}
Let us set $N \coloneqq \lfloor (C \mathsf{p}_\q t_k)/((1+\varepsilon) 2P_{t_k}) \rfloor$ to simplify the notation. We take $\alpha = \varepsilon/2$, then on the event $\{P_{t_k}/t_k \in[\delta, \delta^{-1}]\}$, we have for all $k$ large enough:
\[\frac{\Delta t_k}{N} 
\ge \frac{(1+\varepsilon) 2P_{t_k}}{C \mathsf{p}_\q t_k} \cdot (C+o(1)) \frac{t_k}{\log t_k}
\ge \frac{2P_{t_k}}{\mathsf{p}_\q \log P_{t_k}} \cdot (1+\alpha).\]
For $k \ge 1$, we define $\theta_{t_k, 0} \coloneqq t_k$ and for every $i \ge 1$, we let $\theta_{t_k, i}$ be the first time $n$ when all the edges on the boundary after $n$ peeling steps are at height at least $H_{t_k}+i$. Appealing to the previous bound for the first inequality, and applying Corollary \ref{cor:time_complete_k_layers} conditionally on $P_{t_k}$ for the second one, for every $\gamma \in (0,1)$, for all $k$ large enough, we have
\begin{align*}
\Pr{\left\{\frac{P_{t_k}}{t_k} \in[\delta, \delta^{-1}]\right\} \cap \left\{H_{t_{k+1}} - H_{t_k} \ge N\right\}}
&= \P\bigg(\left\{\frac{P_{t_k}}{t_k} \in[\delta, \delta^{-1}]\right\} \cap \left\{\theta_{t_k, N} \le \Delta t_k\right\}\bigg).
\\
&\ge \P\bigg(\left\{\frac{P_{t_k}}{t_k} \in[\delta, \delta^{-1}]\right\} \cap \left\{\frac{\theta_{t_k, N}}{N} \le \frac{(1+\alpha) 2P_{t_k}}{\mathsf{p}_\q \log P_{t_k}}\right\}\bigg)
\\
&\ge (1-\gamma) \cdot \Pr{\frac{P_{t_k}}{t_k} \in[\delta, \delta^{-1}]}.
\end{align*}
Note that there is a slight subtlety since at time $t_k$, we may not have all the edges on the boundary at the same height, say, $r \ge 0$: we may have some edges at height $r$ and the others at height $r+1$. However this only decreases the amount of time needed to complete the first layer so the bound still holds. The last bound in the previous display converges then to $\P(\mathsf{p}_\q \Upsilon^\uparrow_1 \in[\delta, \delta^{-1}])$ as $\gamma \downarrow 0$ and $k \to \infty$. Recall the inclusion of events \eqref{eq:simplify_bound_eta}, we conclude that
\[\liminf_{k \to \infty} \Pr{\eta_k \ge -1} \ge \Pr{\mathsf{p}_\q \Upsilon^\uparrow_1 \notin[\delta, \delta^{-1}]} + \Pr{\mathsf{p}_\q \Upsilon^\uparrow_1 \in[\delta, \delta^{-1}]} = 1.\]

Let us set $\eta_k^- = - \min\{\eta_k, 0\}$, then for every $C, \delta > 0$, since $H_n$ is increasing,
\begin{align*}
H_n &\ge \sum_{k=1}^{(\log n)^2/(2C)-1} \left(H_{t_{k+1}} - H_{t_k}\right) \ind{P_{t_k}/t_k \in[\delta, \delta^{-1}]}
\\
&\ge \sum_{k=1}^{(\log n)^2/(2C)-1} C \frac{\mathsf{p}_\q}{2} \frac{t_k}{P_{t_k}} \ind{P_{t_k}/t_k \in[\delta, \delta^{-1}]} - \sum_{k=1}^{(\log n)^2/(2C)-1} \eta_k^-.
\end{align*}
Let us prove that the last bound, rescaled by $(\log n)^{-2}$, converges to $(2 \pi^2)^{-1}$ in $\mathrm{L}^1$ when letting first $n \to \infty$ and then $C \to \infty$ and $\delta \to 0$. On the one hand, from Lemma \ref{lem:approximation_Delta_H_t_k}, for any $C > 0$,
\[\lim_{\delta \downarrow 0} \limsup_{n \to \infty}
\Es{\left|\frac{1}{(\log n)^2} \sum_{k=1}^{(\log n)^2/(2C)} \left(C \frac{\mathsf{p}_\q}{2} \frac{t_k}{P_{t_k}} \ind{P_{t_k}/t_k \in[\delta, \delta^{-1}]}\right) - \frac{1}{2 \pi^2}\right|}
= 0.\]
On the other hand, since for all $i \ge 1$, we have $\eta_i^- \le (C \mathsf{p}_\q)/(2\delta)$, we obtain
\[\Es{\eta_i^-} = \Es{\eta_i^- \ind{\eta_i^- \le 1}} + \Es{\eta_i^- \ind{\eta_i^- > 1}}
\le 1 + \frac{C \mathsf{p}_\q}{2\delta} \cdot \Pr{\eta_i^- > 1},\]
which converges to $1$ as $i \to \infty$ from the previous discussion. In particular, for any $\delta > 0$,
\[\lim_{C \to \infty} \limsup_{n \to \infty}
\frac{1}{(\log n)^2} \sum_{k=1}^{(\log n)^2/(2C)} \Es{\eta_k^-} = 0,\]
which completes the proof.
\end{proof}

We close this section with the proof of Lemma \ref{lem:time_complete_first_layer}.

\begin{proof}[Proof of Lemma \ref{lem:time_complete_first_layer}]
First note that, according to Theorem \ref{thm:convergence_processes_perimeter_volume}, for every sequence $k(\ell) = o(\ell)$, for every $\varepsilon \in (0,1)$,
\[\lim_{\ell \to \infty} \P_\infty^{(\ell)}\left(1-\varepsilon \le \inf_{i \le k(\ell)} \frac{P_i}{\ell} \le \sup_{i \le k(\ell)} \frac{P_i}{\ell} \le 1+\varepsilon\right) = 1.\]
It follows that for every $\varepsilon \in (0,1)$,
\[\lim_{\ell \to \infty} \P_\infty^{(\ell)}\left(\frac{P_{\theta_1}}{\ell} \in [1-\varepsilon, 1+\varepsilon] \;\middle|\; \theta_1 \le (1+\varepsilon) \frac{2 \ell}{\mathsf{p}_\q \log \ell}\right) = 1.\]
Let us prove that for every $\varepsilon \in (0,1)$,
\begin{equation}\label{eq:upper_bound_time_complete_layer}
\lim_{\ell \to \infty} \P_\infty^{(\ell)}\left(\theta_1 \le (1+\varepsilon) \frac{2 \ell}{\mathsf{p}_\q \log \ell}\right) = 1.
\end{equation}
The idea is similar to that of \cite[Lemma 12]{Curien-Le_Gall:Scaling_limits_for_the_peeling_process_on_random_maps}. We perform the peeling by layers, starting with a perimeter $2\ell$. Let $\mathcal{B}$ be the set of edges adjacent to the unique hole at time $0$ of the peeling exploration. For every $n \ge 1$, we let $A_n$ be the number of edges of $\mathcal{B}$ which are not adjacent to the hole at time $n$ of the peeling exploration, i.e. edges of $\mathcal{B}$ which have been ``swallowed'' by the process. With the notation from Section \ref{sec:upper_bound_height}, we have $A_n = 2P_0-D_n$. For every $\ell \ge 1$, we let $\sigma_\ell \coloneqq \inf\{n \ge 1 : A_n \ge 2\ell\}$; then observe that $\theta_1$ and $\sigma_\ell$ have the same law under $\P_\infty^{(\ell)}$.

Let us define another non-decreasing sequence $(A_n')_{n \ge 0}$: set $A_0' = 0$ and the steps $\Delta A_n' \coloneqq A_{n+1}' - A_n'$ for every $n \ge 0$ are independent and defined as follows according to the $n+1$-st peeling step:
\begin{itemize}
\item we set $\Delta A_n' = 1$ if we discover a new face, or if we identify the selected edge with another one to its left;
\item we set $\Delta A_n' = 2k+2$ if we identify the selected edge with another one to its right, and such that the finite hole contains exactly $2k \ge 0$ edges.
\end{itemize}
Clearly $A_n \ge A_n'$ for all $1 \le n \le \sigma_\ell - 1$: the only difference between the two is when a peeling step identifies an edge with another one to its left, then in $A_n'$ we do not count all the edges of $\mathcal{B}$ between them. For every $\ell \ge 1$, we let $\sigma_\ell' \coloneqq \inf\{n \ge 1 : A_n' \ge 2\ell\}$ so $\sigma_\ell \le \sigma_\ell'$.

Recall the law $\P^{(\infty)}$ of a bipartite Boltzmann map with a root-face of infinite degree introduced in Section \ref{sec:half_plane_map}. Under $\P^{(\infty)}$, $\mathcal{B}$ is infinite and the variables $(\Delta A_n')_{n \ge 1}$ are i.i.d., distributed in $\{1\} \cup 2\N$ where, for every $j \ge 1$,
\[\P^{(\infty)}(A_1' = 2j) \enskip=\enskip \frac{\nu(-j)}{2}
\enskip\mathop{\sim}^{}_{j \to \infty}\enskip \frac{\mathsf{p}_\q}{2j^2}.\]
Since $A_n$ is non-decreasing, $(n \log n)^{-1} A_n'$ converges in probability to $\mathsf{p}_\q$ as $n \to \infty$. We conclude that $\ell (\log \ell)^{-1} \sigma_\ell'$ converges in probability to $2/\mathsf{p}_\q$ as $\ell \to \infty$ and so for every $\varepsilon > 0$,
\begin{equation}\label{eq:upper_bound_time_complete_layer_half_plane}
\P^{(\infty)}\left(\sigma_\ell \le (1+\varepsilon) \frac{2 \ell}{\mathsf{p}_\q \log \ell}\right) \cv[\ell] 1.
\end{equation}

Recall that under $\P_\infty^{(\ell)}$, $\theta_1$ has the law of $\sigma_\ell$, \eqref{eq:upper_bound_time_complete_layer} will therefore follow from \eqref{eq:upper_bound_time_complete_layer_half_plane} by comparing the first peeling steps under the laws $\P_\infty^{(\ell)}$ and $\P^{(\infty)}$. Recall from Equations \eqref{eq:qp} and \eqref{eq:qp2} the transition probabilities for a peeling step when the  perimeter of the hole is $2\ell$: $\nu(k-1) h^\uparrow(\ell+k-1)/h^{\uparrow}(\ell)$ to discover a new vertex of degree $2k$, and $\frac{1}{2}\nu(-k-1) h^\uparrow(\ell-k-1)/h^{\uparrow}(\ell)$ to identify to the left, and the same to the right, two edges and ``swallowing'' $2k$ edges. The quantities $\nu(k-1)$ and $\nu(-k-1)/2$ are respectively the probabilities of the corresponding events when the perimeter is infinite (i.e. under $\P^{(\infty)}$, see Section \ref{sec:half_plane_map}); furthermore $\ell$ is the half-perimeter before the peeling step and $\ell + k -1$, respectively $\ell - k -1$, is the half-perimeter after the peeling step. We deduce that for every $\ell, k \ge 1$ and for every sequence of peeling steps $\mathsf{PS}_1, \dots, \mathsf{PS}_k$ which occurs with a positive probability under $\P_\infty^{(\ell)}$, if we denote by $\ell_k$ the half-perimeter of the hole after these $k$ peeling steps, it holds that
\[\frac{\P_\infty^{(\ell)}(\text{the first }k\text{ peeling steps are }\mathsf{PS}_1, \dots, \mathsf{PS}_k)}
{\P^{(\infty)}(\text{the first }k\text{ peeling steps are }\mathsf{PS}_1, \dots, \mathsf{PS}_k)}
= \frac{h^\uparrow(\ell_k)}{h^\uparrow(\ell)}.\]

Choose any sequence of integers $(k(\ell))_{\ell \ge 1}$ satisfying $\ell / \log \ell \ll k(\ell) \ll \ell$. As we observed at the beginning of this proof, for every $\delta > 0$,
\[\lim_{k \to \infty} \P_\infty^{(\ell)}\left((1-\delta) \ell \le \inf_{j \le k(\ell)} P_j \le \sup_{j \le k(\ell)} P_j \le (1+\delta) \ell\right) = 1.\]
Recall that $h^\uparrow$ is increasing and that $h^\uparrow(x) \sim 2\sqrt{x/\pi}$ as $x \to \infty$. For every $\delta > 0$ we may choose $\ell$ large enough, such that for every sequence of peeling steps $\mathsf{PS}_1, \dots, \mathsf{PS}_{k(\ell)}$ which occurs with a positive probability under $\P_\infty^{(\ell)}$ and satisfies the above constraint, we have
\[\sqrt{1-2\delta} 
\le \frac{\P_\infty^{(\ell)}(\text{the first } k(\ell) \text{ peeling steps are }\mathsf{PS}_1, \dots, \mathsf{PS}_{k(\ell)})}
{\P^{(\infty)}(\text{the first } k(\ell) \text{ peeling steps are }\mathsf{PS}_1, \dots, \mathsf{PS}_{k(\ell)})}
\le \sqrt{1+2\delta}.\]
The event $\sigma_\ell \le (1+\varepsilon) \frac{2 \ell}{\mathsf{p}_\q \log \ell}$ is measurable with respect to the first $(1+\varepsilon) \frac{2 \ell}{\mathsf{p}_\q \log \ell} \ll k(\ell)$ peeling steps, \eqref{eq:upper_bound_time_complete_layer} thus follows from \eqref{eq:upper_bound_time_complete_layer_half_plane} and this comparison result.
\end{proof}

\section{Bernoulli percolation}
\label{sec:perco}

In this section, we consider face percolation on infinite maps: given $p \in [0,1]$, every face is ``open'' with probability $p$ and ``closed'' with probability $1-p$, independently of each other. Two open faces are in the same cluster if there exists a path in the dual map going from one to the other visiting only open faces. We shall mostly focus on the map $\Map^{(\infty)}$ with a root-face of infinite degree introduced in Section \ref{sec:half_plane_map}, and only briefly discuss the case of $\Map_\infty$, since the former is slightly simpler to study and exhibits a more interesting behaviour.

We first prove that the boundary of $\Map^{(\infty)}$ (which is the set of edges adjacent to the infinite face) has infinitely many cut-edges. This implies that there is no percolation. In Section \ref{sec:RW_Cauchy}, we obtain new estimates on entrance times in $\Z_-$ of random walks attracted to a Cauchy process. We finally use these results in Section \ref{sec:perco_length_interface} to study more precisely the length of percolation interfaces using another well-designed peeling algorithm.

\subsection{\texorpdfstring{Cut-edges in $\Map^{(\infty)}$ and $\Map_\infty$}%
	{Cut-edges in M\^{}(infinity) and M\_infinity}}
\label{sec:cut_points}

Let us first prove the following striking property of the map $\Map^{(\infty)}$, illustrated in Figure \ref{fig:cut_edges_half_plane}. We call \emph{cut-edges} the edges on the boundary of $ \Map^{(\infty)}$ which separate the root-vertex from $\infty$ (recall that $ \Map^{(\infty)}$ is one-ended).
\begin{thm}\label{thm:cut_edges_half_plane}
Almost surely, in $\Map^{(\infty)}$ there are infinitely many cut-edges.
\end{thm}

\begin{figure}[!ht]
\begin{center}
\includegraphics[width=.4\linewidth]{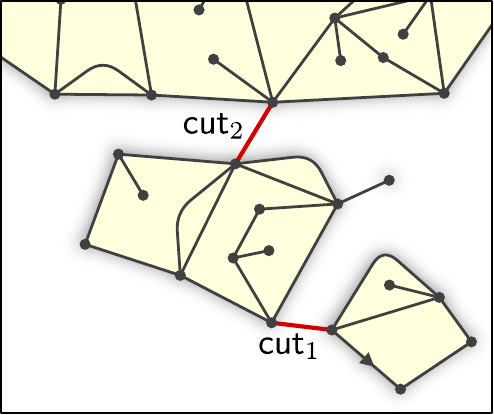}
\caption{Illustration of the decomposition of $\Map^{(\infty)}$.}
\label{fig:cut_edges_half_plane}
\end{center}
\end{figure}

\begin{proof}[Proof of Theorem \ref{thm:cut_edges_half_plane}]
Denote $ \mathsf{cut}_{1}, \mathsf{cut}_{2}, \dots$ the cut-edges ordered as seen from the root-vertex. By the spatial Markov property, conditionally on $ \mathsf{cut}_{1}$ existing, the two maps obtained by removing $ \mathsf{cut}_{1}$ from $\Map^{(\infty)}$ are independent and the infinite one (rooted say at the first edge on the right of $ \mathsf{cut}_{1}$ on its boundary) has the law of $\Map^{(\infty)}$. This clearly implies that the number $\mathscr{C}$ of  cut-edges in $\Map^{(\infty)}$, if finite, has a geometric distribution. Thus, to prove that $ \mathscr{C}=\infty$ almost surely, we only need to prove $ \mathbb{E}[\mathscr{C}] = \infty$. For this, we consider the boundary of $\Map^{(\infty)}$ seen as unexplored, which we identify with $\Z$. By the peeling exploration (recall the transition probabilities from Section \ref{sec:half_plane_map}), for $i \geq 0$ and $ j\geq 0$ such that $i+j$ is even, the probability that the edge $\{(-i-1)\to-i\}$ is identified in $\Map^{(\infty)}$ to the edge $\{j \to (j+1)\}$ is equal to 
\[\frac{1}{2}\nu\left( -\frac{i+j}{2}-1\right).\]
Hence we have  as desired
\[\mathbb{E}[ \mathscr{C}] = \sum_{\begin{subarray}{c}i\ge 0, j \geq 0\\ i+j \mathrm{\ even} \end{subarray} } \frac{1}{2}\nu\left( -\frac{i+j}{2}-1\right) =\infty,\]
thanks to \eqref{eq:tail_nu}.
\end{proof}

\begin{rem}[Cut-edges in $ \Map_{\infty}$]
\label{rem:cut_edges_full_plane}
The same phenomenon should appear in $\Map_{\infty}$ (although the cut-edges are more rare) and we sketch here the argument which is similar to that of \cite[Section 5.2.1]{Budd-Curien:Geometry_of_infinite_planar_maps_with_high_degrees}, see in particular Figure 10 there. Consider a peeling of $ \Map_{\infty}$ and assume that between time $2^{k}$ and $2^{k+1}$ we discover a new face of degree proportional to $2^{k}$. By Theorem \ref{thm:convergence_processes_perimeter_volume}, this event happens with a probability bounded away from $0$ (uniformly in $k$). When, this happens, we claim that there is a probability of order $ \frac{1}{k}$ that two edges of this face are identified and form a cut-edge of $\Map_{\infty}$. This can for example occur if when performing a peeling by layers along this face we discover an unusually large jump of size $2^{k}$ within the first $ 2^{k}/\log 2^{k}$ steps and that this jump creates a cut-edge (see Figure 10 in \cite{Budd-Curien:Geometry_of_infinite_planar_maps_with_high_degrees}). The last event indeed occurs with probability of order $1/k$. Summarising, this heuristic shows that with a probability of order $1/k$ a cut-edge occurs between the peeling times $2^{k}$ and $2^{k+1}$. Since the harmonic series diverges and these events are roughly independent, this should imply existence of infinitely many cut-edges in $\Map_{\infty}$.
\end{rem}

\subsection{Entrance times of random walks attracted to a Cauchy process}
\label{sec:RW_Cauchy}

Our aim in this sub-section is to prove Proposition \ref{prop:intro_tail_first_ladder_time_Cauchy_walk} on the tail behaviour of entrance time for random walks attracted to a symmetric or asymmetric Cauchy process. We shall use then these results to study face percolation on infinite Boltzmann random maps. Let us first prove an easy concentration estimate for non-decreasing random walks in the domain of attraction of a $1$-stable law.

\begin{lem}\label{lem:deviation_increasing_Cauchy_walk}
Let $\xi$ be a random variable on $\Z_+$ satisfying $\P(\xi > k) \sim c k^{-1}$ as $k \to \infty$ for some $c > 0$. Let $W$ be a random walk started from $0$ and with i.i.d. steps distributed as $\xi$. Then for every $\delta > 0$, as $n \to \infty$,
\[\Pr{\left|(n \log n)^{-1} W_n - c\right| \ge \delta} = O\left(\frac{1}{\log n}\right).\]
\end{lem}

It is well-known that such a random walk satisfies $(n \log n)^{-1} W_n \to c$ in probability as $n \to \infty$, and we have used this fact in the proof of Lemma \ref{lem:time_complete_first_layer}, but we shall need these more precise deviation probabilities.

\begin{proof}
Fix $\delta > 0$ small. Let us write $W_n = \xi_1 + \dots + \xi_n$, where the $\xi_i$'s are i.i.d., distributed as $\xi$. We set also $\xi^{(n)} = \xi \cdot \ind{\xi \le n \log n}$, and then $W_n^{(n)} = \xi^{(n)}_1 + \dots + \xi^{(n)}_n$. First, we have the easy upper bound
\[\Pr{\left|(n \log n)^{-1} W_n - c\right| \ge \delta}
\le \Pr{\sup_{1\le i\le n} \xi_i > n \log n} + \Pr{\left|(n \log n)^{-1} W^{(n)}_n - c\right| \ge \delta}.\]
The first term in the right-hand side is asymptotically equivalent to $c/\log n$ as $n \to \infty$ so we focus on the second one. Let us consider the first two moments of $\xi^{(n)}$ as $n \to \infty$:
\[\Es{\xi^{(n)}} \sim c \log n,
\qquad\text{and}\qquad
\Es{(\xi^{(n)})^2} \sim c n \log n,
\qquad\text{so}\qquad
\Var{(\xi^{(n)})^2} \sim c n \log n.\]
Then for all $n$ large enough, we have $|(n \log n)^{-1} \E[W^{(n)}_n] - c| \le \delta/2$ and so
\begin{align*}
\Pr{\left|(n \log n)^{-1} W^{(n)}_n - c\right| \ge \delta}
&\le \Pr{(n \log n)^{-1} \left|W^{(n)}_n - \Es{W^{(n)}_n}\right| \ge \delta/2}
\\
&\le \frac{4}{\delta^2 (n \log n)^2} \cdot \Var{W^{(n)}_n}
\\
&\sim \frac{4}{\delta^2 (n \log n)^2} \cdot cn^2 \log n,
\end{align*}
and the claim follows.
\end{proof}

We next prove Proposition \ref{prop:intro_tail_first_ladder_time_Cauchy_walk} appealing to this lemma.

\begin{proof}[Proof of Proposition \ref{prop:intro_tail_first_ladder_time_Cauchy_walk}]
We start with the symmetric case $c_+=c_-$.
According to \cite[Theorem 8.3.1]{Bingham-Goldie-Teugels:Regular_variation} $W_n/n - b_n$ converges in distribution as $n\to\infty$ to a symmetric Cauchy random variable $\mathcal{C}$ if we choose $b_n = \Es{\frac{W_1}{1+\left(W_1/n\right)^2}}$.
Hence, under the assumption that $b_n = b + o(1)$ as $n\to\infty$ we have
\[\Pr{ W_n > 0} = \Pr{ \frac{1}{n}W_n - b_n > - b_n } = \Pr{\mathcal{C} > -b} + o(1). \] 
Using that the Cauchy random variable $\mathcal{C}$ with tail $\Pr{ \mathcal{C} > x } \sim c_+/x$ has distribution function $\Pr{ \mathcal{C} < x } = \frac{1}{2} + \frac{1}{\pi} \arctan(\frac{x}{\pi c_+})$ we find that $\Pr{W_n>0} = \rho + o(1)$.
In particular, $W_n$ satisfies Spitzer's condition $\frac{1}{n}\sum_{k=1}^n \Pr{W_n>0} \to \rho \in(0,1)$ as $n\to\infty$.
Now \cite[Theorem 8.9.12]{Bingham-Goldie-Teugels:Regular_variation} implies that $\tau$ is in the domain of attraction of a positive stable random variable with index $\rho$, and (iii) follows.

Consider next the second case $c_+ < c_-$. Let $(\xi_i)_{i \ge 1}$ be a sequence of i.i.d. random variables distributed as $W_1$ and let us set
\[W_n^{(+)} \coloneqq \sum_{i=1}^n \xi_i \cdot \ind{\xi_i \ge 0},
\qquad\text{and}\qquad
W_n^{(-)} \coloneqq \sum_{i=1}^n -\xi_i \cdot \ind{\xi_i \le 0},
\qquad\text{so}\qquad
W_n = W_n^{(+)} - W_n^{(-)}.\]
Choose $A \in (c_+, c_-)$, then appealing to Lemma \ref{lem:deviation_increasing_Cauchy_walk}, we obtain
\begin{align*}
\Pr{W_n > 0} &\le \Pr{W_n^{(+)} > A n \log n \text{ or } W_n^{(-)} < A n \log n}
\\
&\le \Pr{W_n^{(+)} > A n \log n} + \Pr{W_n^{(-)} < A n \log n}
\\
&= O\left(\frac{1}{\log n}\right).
\end{align*}
In particular this allows us to bound $\Pr{W_n \geq 0}$ uniformly by the coefficients of an appropriate analytic function.
To be precise, using the transfer theorem from \cite[Theorem 3A]{Flajolet-Odlyzko:Singularity_analysis_of_generating_functions}, there exists a $C>0$ such that
\[ \Pr{W_n \geq 0} \leq [s^n] f(s),
\quad\text{with}\quad f(s) = C \cdot \left(\frac{1}{(1-s)\log\left(\frac{1}{1-s}\right)}-\frac{1}{s}\right).  \]
According to \cite[XII.7 Theorem 4]{Feller:An_introduction_to_probability_theory_and_its_applications_Volume_2} the generating function $p(s)$ of the probabilities 
\[p_n = \Pr{\tau>n} = \Pr{W_1 \geq0, W_2\geq0,\ldots,W_n\geq0}\] 
may be expressed in terms of the probabilities $\Pr{W_m\geq0}$ through
\begin{equation}\label{eq:genfunrelation}
 p(s) = \sum_{n=0}^\infty p_n s^n = \exp\left(\sum_{m=1}^\infty \frac{s^m}{m}\Pr{W_m\geq0}\right).
 \end{equation}
Since the coefficient of $s^n$ of the right-hand side is monotone in $\Pr{W_m\geq 0}$ we have that
\[ p_n \leq [s^n]\exp\left(\sum_{m=1}^\infty \frac{s^m}{m}[s^m]f(s)\right) = [s^n]\exp\left(\int_0^s f(t)\mathrm{d}t\right) = [s^n]\left(\frac{1}{s} \log\left(\frac{1}{1-s}\right) \right)^C. \]
Assuming we have not taken $C$ to be an integer, we may deduce again from the transfer theorem \cite{Flajolet-Odlyzko:Singularity_analysis_of_generating_functions}, Theorem 3A and the second remark at the end of Section 3 that the right-hand side is $O((\log n)^{C-1}/n)$.
Hence, the same is true for $p_n$, giving the correct upper-bound.

For the lower-bound it suffices to consider the situation in which the walk makes a large positive jump at the first step: with the same notation as above and $A > c_-$,
\begin{align*}
\Pr{\tau > n } &\geq \Pr{W_1 > A\, n\log n\text{ and } \tau> n}\\
&\geq \Pr{W_1 > A\, n\log n}\cdot\Pr{W_{n}^{(-)} < A\, n\log n} \\
&\sim \frac{c_+}{A\,n \log n}.
\end{align*}

Consider finally the first case $c_+ > c_-$. Then we may apply the previous results to the walk $-W$ and we have
\[\Pr{W_n \geq 0} = 1 - \Pr{W_n < 0} \ge 1 - O\left(\frac{1}{\log n}\right).\]
Hence there exist $C,N>0$ such that for all $n>N$
\[\Pr{W_n \geq 0} \geq 1- [s^n]f(s) \geq 0.
\]
Similarly as before, by monotonicity we may bound the coefficients of (\ref{eq:genfunrelation}) by
\[
p_n \geq [s^n]\exp\left(\sum_{m=N+1}^\infty \frac{s^m}{m}(1-[s^m]f(s))\right) = [s^n]G(s),
\]
where
\[G(s) = \frac{1}{1-s}\exp\left(- \int_0^s f(t)\mathrm{d}t + P(s)\right) = \frac{1}{1-s} \left(\frac{1}{s} \log\left(\frac{1}{1-s}\right) \right)^{-C} e^{P(s)}\]
with $P(s)= \sum_{m=1}^N \frac{s^m}{m}(1-[s^m]f(s))$ a polynomial of degree $N$.
The transfer theorem \cite[Theorem 3A]{Flajolet-Odlyzko:Singularity_analysis_of_generating_functions} applies to $G(s)$, giving
\[ p_n \geq [s^n]G(s) \sim \ex^{P(1)} (\log n)^{-C}, \]
and thus completing the proof.
\end{proof}

\begin{rem}\label{rem:polylog}
	We expect that under the assumptions of Proposition \ref{prop:intro_tail_first_ladder_time_Cauchy_walk}, the walk $W$ satisfies
	\begin{equation}\label{eq:Wposprob}
	\begin{aligned}
  1-\Pr{W_n > 0} &\sim \frac{c_-}{c_+-c_-} (\log n)^{-1} \qquad\text{when}\qquad c_+>c_-,
  \\
  \Pr{W_n > 0} &\sim \frac{c_+}{c_--c_+} (\log n)^{-1} \qquad\text{when}\qquad c_+<c_-.
\end{aligned}
	\end{equation} 
	When this is the case one may easily adapt the proof of Proposition \ref{prop:intro_tail_first_ladder_time_Cauchy_walk} to obtain
	\[ \Pr{\tau \geq n} = \begin{cases} (\log n)^{-\frac{c_-}{c_+-c_-} + o(1)} & \text{when }c_+>c_- \\
	n^{-1}(\log n)^{\frac{c_+}{c_--c_+}-1 + o(1)} & \text{when }c_+<c_-
	\end{cases}.\]
	It turns out (\ref{eq:Wposprob}) can be established straightforwardly under slightly stronger assumptions on the law of $W$. 
	Indeed, if we assume that $\Pr{ W_1 > k } = c_+ k^{-1} + o(k^{-1}(\log k)^{-2} )$ and $\Pr{ W_1 < -k } = c_- k^{-1} + o(k^{-1}(\log k)^{-2} )$ as $k\to\infty$, then according to Hall \cite[Theorem 3]{Hall:Two_sided_bounds_on_the_rate_of_convergence_to_a_stable_law} the distribution function $x\mapsto\Pr{ a_n^{-1}W_n-b_n < x}$ converges to the distribution $F_\infty(x)$ of an asymmetric Cauchy random variable with a uniform error of size $o(1/\log n)$.
	Using \cite[Theorem 8.3.1]{Bingham-Goldie-Teugels:Regular_variation} we may choose $a_n = n$ and $b_n \sim (c_+-c_-)\log n$, and therefore 
	\[\Pr{W_n>0} = 1- \Pr{ a_n^{-1}W_n - b_n \leq -b_n} = 1-F_{\infty}(-b_n) + o(1/\log n). \]
	Since $F_\infty(-x) \sim c_-/x$ and $1-F_{\infty}(x) \sim c_+/x$ as $x\to\infty$, we indeed obtain (\ref{eq:Wposprob}).
\end{rem}

\subsection{\texorpdfstring{On the length of the interfaces in $\Map^{(\infty)}$}%
	{On the length of the interfaces in M\^{}(infinity)}}
\label{sec:perco_length_interface}

We close this paper by considering face percolation on the map $\Map^{(\infty)}$ in which each face is coloured independently black with probability $p \in (0,1]$. This corresponds to site percolation on the dual maps. In both cases, we shall adopt the following boundary condition: we impose that the infinite face is white, and we ``unzip'' the root-edge to create a face of degree two that we colour black.

Recall the decomposition of $\Map^{(\infty)}$ as a string of disjoint finite maps obtained in Theorem \ref{thm:cut_edges_half_plane}. Since each black cluster is contained in such a finite part, this readily implies that for any $p \in (0,1]$, almost surely there is no infinite black cluster in $\Map^{(\infty)}$. Similarly, the result in Remark \ref{rem:cut_edges_full_plane} would imply that for any $p \in (0,1)$, almost surely there is no infinite black cluster in $\Map_\infty$.

Although there is no percolation in $\Map^{(\infty)}$, even at $p=1$, we next prove that the behaviour of the length of percolation interfaces changes at $p=1/2$. More precisely, consider the black cluster containing the root-edge, which we view as a black face of degree two; we denote by $\mathscr{L}$ the number of edges on the outer-most boundary of the latter, as depicted in Figure \ref{fig:boundary_perco_cluster}.

\begin{prop}\label{prop:number_edges_boundary_perco_half_plane}
As $n \to \infty$, we have
\[\Pr{\mathscr{L} > n} = n^{-\lambda(p) + o(1)}
\qquad\text{where}\qquad
\lambda(p) = 
\begin{cases}
1 &\text{when }p < 1/2,\\
\frac{1}{\pi}\arctan(2\pi \mathsf{p}_\q) &\text{when }p = 1/2,\\
0 &\text{when }p > 1/2.
\end{cases}\]
\end{prop}

The proof relies on the following peeling process $(\overline{\mathfrak{e}}_i)_{i \ge 0}$ which follows the percolation interfaces, defined in \cite{Curien:Peccot} (see also \cite{Angel-Curien:Percolations_on_random_maps_half_plane_models, Richier:Universal_aspects_of_critical_percolation_on_random_half_planar_maps} for related models).  For each $i \ge 0$, each edge on the boundary of the unique hole of $\overline{\mathfrak{e}}_i$ is given the colour of the face incident to it in $\overline{\mathfrak{e}}_i$. We assume that $\overline{\mathfrak{e}}_0$ consists of the white infinite face and the root-edge, which we view as a black face of degree $2$. Consider next the following hypothesis for a given step $i \ge 0$:

\begin{center}
\begin{minipage}{.9\linewidth}
$(H)$: the set of black edges on the boundary of $\overline{\mathfrak{e}}_i$ is connected, possible empty.
\end{minipage}
\end{center}

Fix $p \in (0,1]$. At the initial step $\overline{\mathfrak{e}}_0$, there is only one black edge on the boundary. If $\overline{\mathfrak{e}}_i$ satisfies $(H)$ and contains $\ell \ge 1$ black edges, then the next edge to peel $\mathcal{A}(\overline{\mathfrak{e}}_i)$ is the left-most black edge. We then face several possible situations when peeling this edge:
\begin{enumerate}
\item either we discover a new black face of degree $2k$ for some $k \ge 1$, this occurs with probability $p \cdot \nu(k-1)$;
\item or we discover a new white face, this occurs with probability $(1-p) \cdot \nu([0, \infty))$;
\item or the black edge is identified with an edge on its left, this occurs with probability $\nu((-\infty, -1])/2$;
\item or, if $\ell \ge 2$, it is identified with another black edge on its right, thus creating a hole of perimeter $0 \le 2k \le \ell-2$, this occurs with probability $\nu(-k-1)/2$;
\item or it is identified with a white edge on its right, this occurs with probability $\nu((-\infty, -\frac{\ell-1}{2}-1])/2$.
\end{enumerate}
In all cases, $\overline{\mathfrak{e}}_{i+1}$ satisfies $(H)$. In the first and third cases, or in the second and fourth, if any black edges remain on the boundary, the next edge to peel $\mathcal{A}(\overline{\mathfrak{e}}_{i+1})$ is again the left-most black edge on the boundary. In the last case, or in the second or fourth if there is no black edge remaining on the boundary, then we have fully discovered a black cluster, and the next edge to peel $\mathcal{A}(\overline{\mathfrak{e}}_{i+1})$ is the first white edge immediately to its right.

\begin{figure}[!ht]
\begin{center}
\includegraphics[width=.9\linewidth]{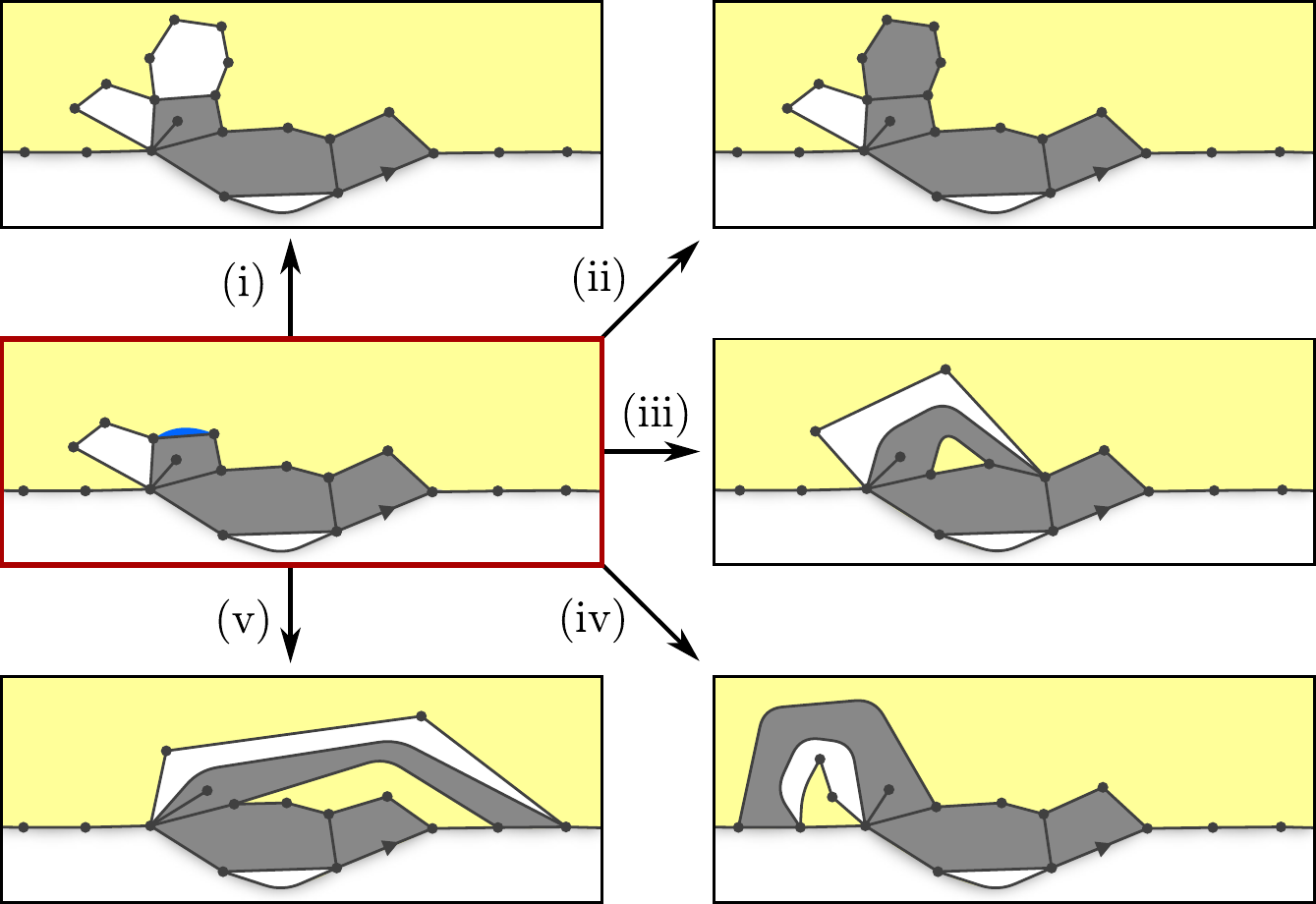}
\caption{The map on the left in the red frame depicts an explored region satisfying hypothesis $(H)$. The five other figures correspond to the five different situations when peeling the edge indicated in blue.  }
\label{fig:perco_half_plane}
\end{center}
\end{figure}

\begin{proof}[Proof of Proposition \ref{prop:number_edges_boundary_perco_half_plane}]
Let us consider the exploration process $(\overline{\mathfrak{e}}_i)_{i \ge 0}$ in which we select the left-most edge on the boundary of the black cluster. We let $\theta \in \N \cup\{\infty\}$ be the first instant at which there is no black edge on the boundary and for all $0 \le n \le \theta$, we let $P^\bullet_n$ be the number of black edges on the boundary of $\overline{\mathfrak{e}}_n$, so $P^\bullet_0 = 1$ and $P^\bullet_\theta = 0$. Let $X$ be a random variable with the following distribution: for $k \ge 0$,
\begin{equation}\label{eq:transition_proba_perco}
\begin{aligned}
\Pr{X = 2k} &= p \cdot \nu(k),\\
\Pr{X = -1} &= (1-p) \cdot \nu([0, \infty)) + \frac{\nu((-\infty, -1])}{2},\\
\Pr{X = -2k-2} &= \frac{\nu(-k-1)}{2}.
\end{aligned}
\end{equation}
We define a random walk $W$ started from $0$ and with i.i.d. steps distributed as $X$ and we set $\tau = \inf\{n \ge 1 : W_n \le -1\}$. Then $\theta$ has the law of $\tau$, and more precisely the sequence $(P^\bullet_0, \dots, P^\bullet_{\theta-1})$ has the law of $(W_0+1, \dots, W_{\tau-1}+1)$. Note that the random walk $W$ oscillates for all $p$, so $\tau$, and therefore $\theta$, is finite almost surely.

From \eqref{eq:tail_nu} we deduce that the walk $W$ satisfies 
\[\Pr{W_1>k} \sim \frac{c_+}{k}
\qquad\text{and}\qquad\Pr{W_1<-k} \sim \frac{c_-}{k} 
\qquad\text{as}\qquad k \to \infty.\]
with $c_+ = 2 p\cdot \mathsf{p}_\q$ and $c_- = \mathsf{p}_\q$.
Hence, when $p\neq 1/2$ Proposition \ref{prop:intro_tail_first_ladder_time_Cauchy_walk} already implies that $\Pr{\theta > n} = n^{-\lambda(p) + o(1)}$.
For $p=1/2$ we still need to consider the limit of $\E[W_1/(1+(W_1/n)^2)]$ as $n\to\infty$.
Since $\Pr{W_1 = -1}=1/2$ and $\Pr{W_1 = 2k} = \Pr{S_1=k}/2$ for $k\in\Z$ with $S$ the walk with increments of law $\nu$ started from $0$, we have
\[\Es{\frac{W_1}{1+\left(W_1/n\right)^2}} = \frac{1}{2} \cdot \frac{-1}{1+n^{-2}} + \Es{\frac{S_1}{1+\left(2S_1/n\right)^2}}.
\]
According to Proposition \ref{prop:convergence_walk_Cauchy} the walk $S$ is centred and therefore the second term is $o(1)$ as $n\to\infty$ (see the beginning of the proof of Proposition \ref{prop:intro_tail_first_ladder_time_Cauchy_walk}). 
Hence the hypothesis of Proposition \ref{prop:tail_first_ladder_time_Cauchy_walk_symmetric} is satisfied by $W$ with $b=-1/2$, and $\Pr{\theta > n} = n^{-\lambda(p) + o(1)}$ follows.

\begin{figure}[!ht]
\begin{center}
\includegraphics[width=.9\linewidth]{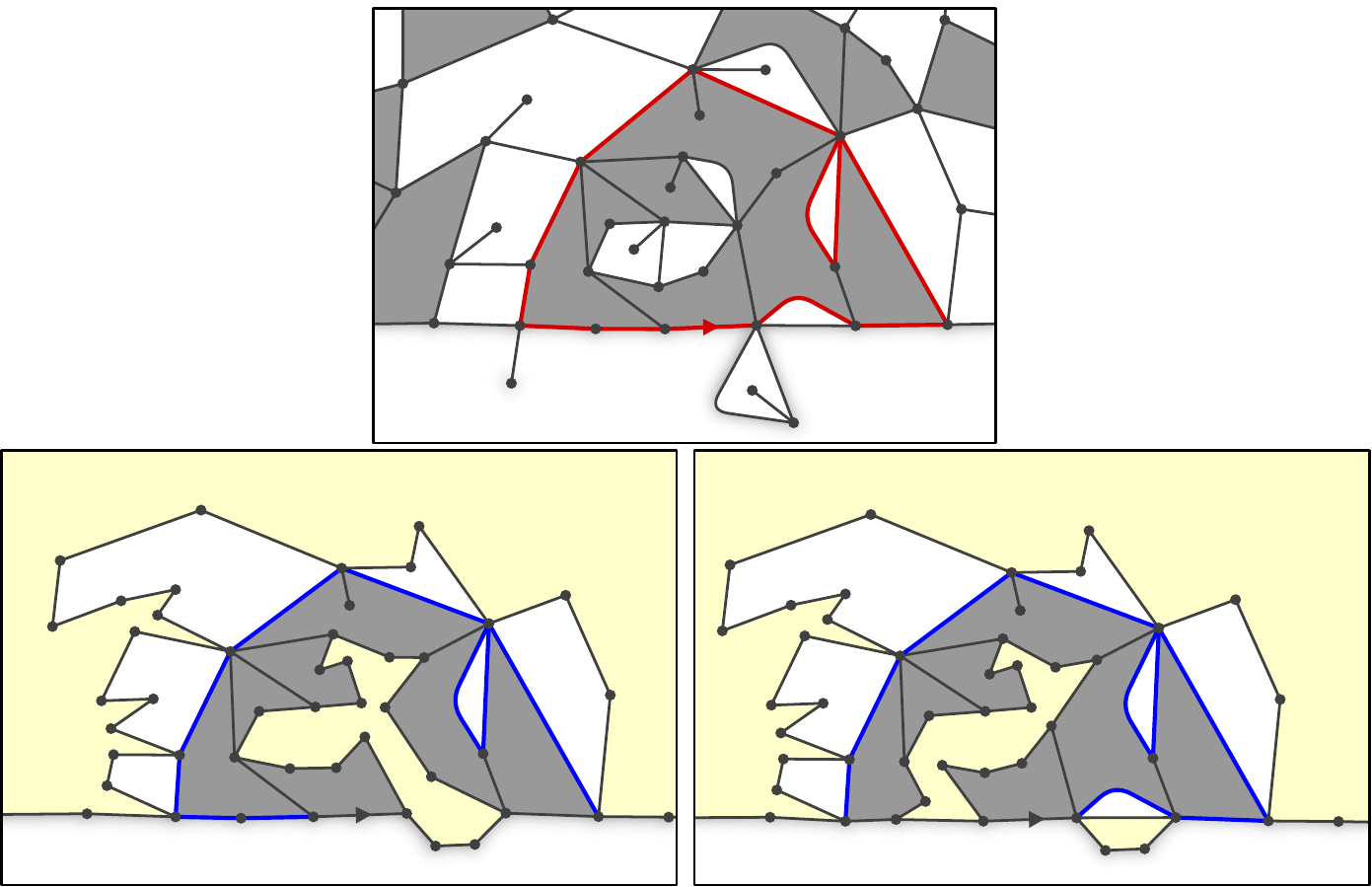}
\caption{Illustration of the outer-most boundary of the percolation cluster: $\mathscr{L} = 12$ is the number of edges on the red cycle and $\mathscr{N} = 9$ and $\mathscr{N}' = 9$ are the number of edges on the blue path on the left and on the right respectively.}
\label{fig:boundary_perco_cluster}
\end{center}
\end{figure}

We now aim at comparing the tail behaviour of $\mathscr{L}$ with that of $\theta$; we introduce an auxiliary random variable $\mathscr{N}$ as follows. Consider the stopped exploration $(\overline{\mathfrak{e}}_i)_{0 \le i < \theta}$ and for each $i < \theta$, let $\mathscr{N}_i$ denote the number of peeling steps among the first $i$ at which either we have discovered a new white face, or we have identified the selected edge with another one adjacent to a white face on the left. Let us set $\mathscr{N} = \mathscr{N}_{\theta-1}$, then the pair $(\theta, \mathscr{N})$ has the same law as $(\tau, \#\{1 \le k < \tau : W_k = W_{k-1}-1\})$. Observe that $\mathscr{N} \le \theta$. Let us set $\varphi(p) \coloneqq \P(W_1 = -1) =  (1-p) \nu([0, \infty)) + \nu((-\infty, -1])/2$ and fix a constant $C > 2/\varphi(p)$. We then write
\begin{align*}
\Pr{\mathscr{N} > n} &\ge \Pr{\mathscr{N} > n \text{ and } \theta > Cn}
\\
&\ge \Pr{\mathscr{N}_{Cn} > n \text{ and } \theta > Cn}
\\
&\ge \Pr{\theta > Cn} - \Pr{\mathscr{N}_{Cn} \le n \text{ and } \theta > Cn}
\\
&\ge \Pr{\theta > Cn} - \Pr{\mathscr{N}_{Cn} \le n}.
\end{align*}
Since $\mathscr{N}_{Cn}$ has the binomial distribution with parameters $Cn$ and $\varphi(p) > 2/C$, the Chernoff bound thus yield
\[\Pr{\mathscr{N}_{Cn} \le n}
\le \Pr{\mathscr{N}_{Cn} \le Cn \left(\varphi(p) - \frac{1}{C}\right)}
\le \exp\left(-2n/C\right)
= o\left(\Pr{\theta > Cn}\right).\]
It follows that
\[\Pr{\mathscr{N} > n} \ge \Pr{\theta > Cn} \cdot (1+o(1)) = n^{-\lambda(p) + o(1)}.\]
We conclude that $\P(\mathscr{N} > n) = n^{-\lambda(p) + o(1)}$.

Consider the exploration process defined similarly but now we select the right-most edge on the boundary of the black cluster at each step, and we denote by $\mathscr{N}'$ the number of steps at which either we have discovered a new white face, or we have identified the selected edge with another one adjacent to a white face on the right, until there is no black edge adjacent to the hole. Clearly, we have $\mathscr{L} \ge \sup\{\mathscr{N}, \mathscr{N}'\}$ and simple geometric considerations show that $\mathscr{L} \le \mathscr{N} + \mathscr{N}$: the two parts of length $\mathscr{N}$ and $\mathscr{N}'$ meet to form the outer-most boundary of length $\mathscr{L}$, see Figure \ref{fig:boundary_perco_cluster} for an illustration. The claim follows from the previous discussion after observing that $\mathscr{N}$ and $\mathscr{N}'$ have the same law.
\end{proof}

%%%%%%%%%%%%%%%%%%%%%%%%%%%%%%%%%%%%%%%%%%%%%%%%%%%%%
%%%%%%%%%%%%%%%%%%%%%%%%%% BIBLIOGRAPHIE %%%%%%%%%%%%%%%%%
%%%%%%%%%%%%%%%%%%%%%%%%%%%%%%%%%%%%%%%%%%%%%%%%%%%%%

%\bibliographystyle{acm}
%\bibliography{Bibliographie}

\end{document}